\documentclass[pdflatex,sn-mathphys-num]{sn-jnl}

\usepackage{graphicx}
\usepackage{multirow}
\usepackage{amsmath,amssymb,amsfonts}
\usepackage{amsthm}
\usepackage{mathrsfs}
\usepackage[title]{appendix}
\usepackage{xcolor}
\usepackage{textcomp}
\usepackage{manyfoot}
\usepackage{booktabs}
\usepackage{algorithm}
\usepackage{algorithmicx}
\usepackage{algpseudocode}
\usepackage{listings}

\usepackage[all]{xy}
\usepackage{enumerate}
\usepackage{epic}
\usepackage{ulem}

\newtheorem{theorem}{\sc Theorem}[section]
\newtheorem{lemma}[theorem]{\sc Lemma}
\newtheorem{remark}[theorem]{\sc Remark}
\newtheorem{example}[theorem]{\sc Example}
\newtheorem{definition}[theorem]{\sc Definition}
\newtheorem{question}{\scshape Question}

\newtheorem{proposition}[theorem]{\sc Proposition}
\newtheorem{corollary}[theorem]{\sc Corollary}

\newtheorem{theoremalfa}{Theorem}

\begin{document}

\title{Profinite groups whose closed subgroups are subnormal of defect at most two}

\author[1]{\fnm{Anderson} \sur{Porto}}\email{ander.porto@ict.ufvjm.edu.br}
\affil[1]{\orgdiv{ICT}, \orgname{Federal University of the Jequitinhonha and Mucuri Valleys}, \orgaddress{\city{Diamantina}, \country{Brazil}}}

\equalcont{These authors contributed equally to this work.}

\author*[2]{\fnm{Igor} \sur{Lima}}\email{igor.matematico@gmail.com}

\equalcont{These authors contributed equally to this work.}

\author[2]{\fnm{Millena} \sur{Silva}}\email{millena\_andrade10@hotmail.com}
\affil[2]{\orgdiv{Department of Mathematics}, \orgname{University of Brasilia}, \orgaddress{\city{Brasilia}, \country{Brazil}}}

\equalcont{These authors contributed equally to this work.}

\abstract{In this paper, we classify profinite groups whose closed subgroups are subnormal of defect at most two, providing profinite versions of theorems due to Heineken and to Mahdavianary. We also establish that if $G$ is a finitely generated periodic pro-$\pi$ group, where $\pi$ is a finite set of odd primes, and all its procyclic subgroups are subnormal of defect at most $2$, then $G$ is a finite $\pi$-group in which every abelian subgroup is subnormal of defect at most $2$.
}

\keywords{Generalized Dedekind groups, Profinite groups, Mahdavianary's Theorem, Subnormal defect.} 
\pacs[MSC Classification]{20E34, 20F19, 20F45.}

\maketitle

\section{Introduction}

A subgroup $H$ of a group $G$ is said to be \textit{subnormal} in $G$ if there exists a finite chain of subgroups
$$ H=G_{m}\unlhd G_{m-1} \unlhd \cdots \unlhd G_{1} \unlhd G_{0} = G \quad \textbf{ $(\star)$ } $$
for some integer $m \geq 0$. If $H$ admits such subnormal chain, we say that $H$ is \textit{subnormal of defect at most $m$}, and write $H \unlhd^{m} G$. Among all chains of the form \textbf{$(\star)$} from $H$ to $G$, the minimum length is called the \textit{subnormal defect} (or simply the \textit{defect}) of $H$ in $G$. Note that normal subgroups are precisely those of defect at most $1$.

The systematic study of subnormal subgroups originated with Wielandt \cite{Wielandt}, who established the foundational theory for finite groups. The concept of defect was later formalized to analyze the length of subnormal chains, with significant contributions from Roseblade \cite{rose} and Möhres \cite{moehres} regarding infinite groups.

It is a classical result that if a group $G$ is nilpotent of class $c$, then every subgroup of $G$ is subnormal of defect at most $c$ (see \cite[p.~154]{hall 2}). Roseblade \cite{rose} established a strong converse to this theorem: any group in which every subgroup is subnormal of defect at most $d$ is necessarily nilpotent, and its nilpotency class is bounded by a function $\mu(d)$ depending exclusively on $d$. The bounding condition about the subnormal defect in Roseblade's Theorem cannot be relaxed to merely requiring all subgroups to be subnormal. Indeed, Heineken and Mohamed \cite{HeineMoha} constructed a group in which every subgroup is subnormal, yet the center is trivial, which implies that the group cannot be nilpotent.

Let $H$ be a closed subgroup of a profinite group $G$. We say that $H$ is \textit{subnormal} in $G$ if there exists a chain as in $(\star)$ in which every intermediate subgroup is closed in $G$. In the profinite context, the concept of subnormal defect coincides with the abstract definition, as long as the subgroups in the chain are closed (see \cite{dusautoy}). We write $H \unlhd^{d}_{c} G$ if $H$ is a closed subgroup of $G$ of subnormal defect at most $d$, where the intermediate subgroups are also closed.

Clearly, $G$ itself is the unique subgroup of defect $0$. In an abstract Dedekind group, every subgroup is normal, hence it has subnormal defect at most $1$ (see \cite{porto, lennox}). Moreover, any group in which every subgroup has defect at most \(1\) is nilpotent of class at most \(2\), and this bound is sharp; in Roseblade's terminology, this is the content of \(\mu(1)=2\). However, the converse fails: note that the dihedral group \(D_8\) is nilpotent of class \(2\) yet possesses a non-normal subgroup with defect \(2\). If $G$ is a profinite Dedekind group, the analogous property holds for all its closed subgroups (see \cite{PB19}). Furthermore, it is known that non-normal (closed) subgroups of (profinite) Cappitt groups have subnormal defect exactly $2$ (see \cite{Por19, Cap71}). Note that the only torsion-free (profinite) Dedekind groups are (profinite) abelian, and hence have nilpotency class equal to 1. 

The structure of groups whose subgroups have subnormal defect at most 2 has been thoroughly investigated by authors such as Mahdavianary \cite{ mahdavinilpotente, mahdavimain}, Cappitt \cite{cappitt2}, and Heineken \cite{heineken}. A fundamental result, valid for both abstract and profinite groups, states that if every cyclic (respectively, procyclic) subgroup is normal in a group (respectively, profinite group) $G$, then every (closed) subgroup of $G$ is normal; that is, $G$ is a (profinite) Dedekind group (see \cite{porto, PB19}). However, for groups in which all subgroups are 2-subnormal, this equivalence does not hold in general. Mahdavianary proved that if a group $G$ is non-periodic, or periodic with no elements of even order, then every cyclic subgroup of $G$ is 2-subnormal if and only if every abelian subgroup of $G$ is 2-subnormal (\cite[Theorem~1(i)]{mahdavimain}). For 2-generated groups, this equivalence holds unconditionally (see \cite{mahdavithesis}). Several other classes of groups are discussed in Mahdavianary's work, including examples proving that the corresponding group classes are strictly distinct (\cite[Theorem~1(ii)--(iii)]{mahdavimain}). 

It is known that a group in which every cyclic subgroup is $2$-subnormal is nilpotent of class at most $3$ (see Theorem 1 in \cite{mahdavinilpotente}). Consequently, the same bound holds for any group in which every subgroup is $2$-subnormal, since in this case every cyclic subgroup is in particular $2$-subnormal. In Roseblade's terminology, this is the content of $\mu(2)=3$. If such a group is torsion-free, its nilpotency class is at most $2$ (see \cite{kappetraus} or \cite{stadelmann} ). However, this condition fails for higher defects. The requirement that every cyclic subgroup be $3$-subnormal is insufficient to guarantee nilpotency, in contrast to the case where every subgroup is $3$-subnormal, which by Roseblade's theorem implies nilpotency of class at most $\mu(3)$. Traustason \cite{traust} established that, for groups which are either torsion-free or of prime exponent different from $7$, the property that every subgroup is $3$-subnormal is equivalent to nilpotency of class at most $3$. In contrast, for groups of exponent $7$, the same condition still implies nilpotency, but with class at most $4$, and this bound is sharp.  More recently, Smith and Traustason \cite{SmiTra} showed that torsion-free groups with all subgroups subnormal of defect at most $4$ are nilpotent of class at most $4$. Taken together, these results reveal a sharp contrast between the behaviour of periodic and non-periodic groups with respect to subnormality restrictions. 

The following problem, due to Traustason (see Problem 16.97 in the Kourovka Notebook \cite{kourovka}), asks whether a torsion-free group all of whose subgroups are subnormal of defect at most \(n\) must have nilpotency class at most \(n\). As discussed in the preceding paragraphs, the answer is known to be positive for \(n \leq 4\). However, a recent construction of Brescia, Di Siena and Trombetti \cite{Brescia2026} provides a 2-generated torsion-free nilpotent group of nilpotency class six in which every subgroup is subnormal of defect at most five, thus answering the case \(n=5\) negatively. Accordingly, the problem reduces to the following.

\begingroup
\setlength{\topsep}{-8pt}
\begin{question}[Traustason -- Problem 16.97, revised]
\label{korou}
Does every torsion-free group all of whose subgroups are subnormal of defect at most \(n\) have nilpotency class at most \(n\), for all \(n \geq 6\)?
\end{question}
\endgroup

An analogous question may be asked for profinite groups, although only for \(n\geq 2\). The restriction \(n\geq 6\) in the abstract setting reflects the fact that the cases \(n=1,2,3,4\) are settled in the literature and the case \(n=5\) has been settled negatively by the construction above, whereas in the profinite context the problem remains open from \(n\geq 2\) onwards (see Question~1). In this paper, we take a step in this direction by settling the case \(n=2\) for torsion-free profinite groups (see Theorem~\ref{leviprofinite}). More precisely, taking a non-periodic profinite group such that all its infinite procyclic subgroups are torsion-free, we prove a profinite version of
Heineken's theorem on \(2\)-Engel groups \cite[Theorem~1(I)]{heineken},
which is used in the proof of Theorem~\ref{maintheorem}, where we establish
a profinite analogue of Mahdavianary's theorem
\cite[Theorem~1]{mahdavimain}. We also show that if $G$ is a finitely generated periodic pro-$\pi$ group, with $\pi$ a finite set of odd primes, and every procyclic subgroup of $G$ is subnormal of defect at most $2$, then $G$ is a finite $\pi$-group in which all its abelian subgroups are subnormal of defect at most $2$. Finally, we investigate the strictness of the inclusions among the classes of profinite groups whose closed, abelian, and procyclic subgroups, respectively, are $2$-subnormal. We construct infinite profinite counterexamples showing that these classes are pairwise distinct, thus extending Mahdavianary's discrete examples to the profinite setting. 

\section{Notation, commutator identities and auxiliary results}

Let $G$ be a group. For an element $x \in G$, we denote its order by $|x|$. If $x$ has infinite order, we write $|x|=\infty$; otherwise, we write $|x|<\infty$. The set of torsion elements (elements of finite order) of $G$ is denoted by $\operatorname{tor}(G)$. For elements $g, h \in G$, the conjugate of $h$ by $g$ is denoted by $h^g = g^{-1}hg$, and the commutator of $h$ and $g$ is defined as $[h, g] = h^{-1}g^{-1}hg = h^{-1}h^g$. More generally, for elements $x_1, x_2, \ldots, x_n \in G$ ($n \geq 2$), the \textit{left-normed} commutator of weight $n$ is defined recursively by the rule $[x_1, x_2, \ldots, x_n] = [[x_1, \ldots, x_{n-1}], x_n]$. When $x_1 = x$ and $x_2 = \cdots = x_n = y$, we abbreviate the commutator $[x, y, \ldots, y]$ as $[x,_m y]$, where $m = n-1$ is the number of times $y$ appears. For $g \in G$, let $\tau_g$ denote the inner automorphism of $G$ defined by $\tau_g(x) = g^{-1}xg$ for all $x \in G.$  

For a subgroup $H$ of a group $G$, denote by $C_G(H)$, $N_G(H)$ and $H^G=\langle h^g \mid h\in H,\ g\in G\rangle$ the centralizer, normalizer and normal closure of $H$ in $G$, respectively.

Let $G$ be a profinite group, and write $\overline{X}$ for the topological closure of a subset $X\subseteq G$. We write $H\leqslant_c G$, $H\unlhd_c G$, $H\leqslant_o G$ and $H\unlhd_o G$ according as $H$ is a closed subgroup, a closed normal subgroup, an open subgroup or an open normal subgroup of $G$; if $H$ is closed, then so are $C_G(H)$ and $N_G(H)$.

Now let $H$ be a subgroup of a residually finite group $G$, with profinite completion $\widehat{G}$. The symbols $\left(\overline{H}\right)^{\widehat{G}}$ and $\overline{H^G}$ denote, respectively, the normal closure of $\overline{H}$ in $\widehat{G}$ and the topological closure in $\widehat{G}$ of the normal closure $H^G$ of $H$ in $G$; thus $\overline{\overline{H}^{\widehat{G}}}$ is the smallest closed normal subgroup of $\widehat{G}$ containing $\overline{H}$. 

Let $G$ be a profinite group. Define the series $(H_i)_{i=0}^{\infty}$ of successive normal closures by
\[
H_0=G,\qquad H_{i+1}=\overline{H^{H_i}}\quad(i\geq 0),
\]
where the closure is taken in $G$. Thus we obtain a descending series
$$H \leq \cdots \unlhd_c H_{i+1} \unlhd_c H_i \unlhd_c \cdots \unlhd_c H_1 \unlhd_c H_0 = G.$$ In fact, this series descends as rapidly as possible among all normal series of closed subgroups of $G$ that contain $H.$ That is, whenever $(K_i)_{i=0}^{\infty}$ is another descending series with $K_0 = G$, $K_{i+1} \unlhd_c K_i$ and $H \subseteq K_i$ for all $i$, we also have that for each $i$, $H_i \subseteq K_i$. Thus the series constructed above yields the fastest possible descent among all normal series of closed subgroups starting at $G$ and containing $H$. The proof follows the abstract argument of \cite[p.~1--2]{lennox}, adapted to the profinite setting by taking topological closures. Consequently, $H$ is subnormal in $G$ if and only if the descending series reaches $H$ at some finite stage, i.e\ $H_n = H$ for some $n \geq 0$. The smallest such $n$ is precisely the subnormal defect of $H$ in $G$. When, in particular, \(H \leq_c G\) has defect at most \(2\), it follows that
$H = H_2 = \overline{H^{H_1}}=\overline{H^{\overline{H^G}}}$,
that is,
$H \unlhd_c \overline{H^G} \unlhd_c G$. The abstract situation can be seen in Chapter~1 of \cite{lennox}.

A (profinite) group $G$ is an \textit{$n$-Engel (profinite) group} if $[x,_n y] = 1$ for all $x, y \in G$. An element $x \in G$ is called a \textit{right $n$-Engel element} of $G$ if $[x,_n g] = 1$ for all $g \in G$. Analogously, $x \in G$ is a \textit{left $n$-Engel element} if $[g,_n x] = 1$ for all $g \in G$. A comprehensive discussion on left-normed commutators and Engel conditions can be found in Robinson \cite[Chapter~7]{chapter7}.

For further background on profinite groups, we refer the reader to Ribes and Zalesskii \cite{RZ10} and Wilson \cite{wilson98}. For the abstract theory of groups, see Robinson \cite{robinson} and Rotman \cite{rotman}.

The following commutator identities will be used throughout the remainder of this paper without explicit reference. Identities \eqref{comut inverso} and \eqref{comut conjugado} are immediate; identity (3) follows from \eqref{comut conjugado}. Identities \eqref{comut produto esq} through \eqref{comut conjug em comut} can be found in Heineken \cite{heineken} and Mahdavianary \cite{mahdavinilpotente}. The Hall-Witt identity is a standard result (see \cite[5.1.5]{robinson}).

\begin{lemma}[Commutator identities] \label{prop comutadores} 
Let $a, b, c$ be elements of a group $G$. Then the following identities hold:
\begin{enumerate}
    \item $[a,b] = 1 \Longleftrightarrow [a,b^{-1}] = 1$; \label{comut inverso}
    \item $b^a = b[b,a]$, $ab = ba[a,b]$, and $[a,b]^c = [a^c,b^c]$; \label{comut conjugado}
    \item $[a^{c^{-1}}, b] = [a, b[b,c]]^{c^{-1}}$; 
    \item $[ab,c] = [a,c]^b[b,c] = [a,c][a,c,b][b,c]$; \label{comut produto esq}
    \item $[a,bc] = [a,c][a,b]^c = [a,c][a,b][a,b,c]$; \label{comut produto dir}
    \item $[a^{-1},b] = [b,a]^{a^{-1}}$ and $[a,b^{-1}] = [b,a]^{b^{-1}}$; \label{comut tirar inverso}
    \item $b^{b^a} = b[b,[b,a]]$; \label{comut conjug em comut}
    \item \textit{(The Hall-Witt identity)} $[a,b^{-1},c]^{b} \cdot [b,c^{-1},a]^c \cdot [c,a^{-1},b]^a = 1$.
\end{enumerate}
\end{lemma}

\begin{remark} \label{aremark} 
Let $H \le G$ be a subgroup. We have $H^G = H \cdot [H,G]$, since $h^g = h \cdot [h,g]$ for all $h \in H$ and $g \in G$. Hence $H^{H^G} = H^{[H,G]} = H^{[G,H]}$. Therefore $H^G$ normalizes $H$ if and only if $[G,H]$ normalizes $H$. Now let $G$ be a profinite group and $H \leq_c G$. Since $H$ is closed, the normalizer $N_G(H)$ is a closed subgroup of $G$. Therefore $[G,H] \subseteq N_G(H)$ if and only if $\overline{[G,H]} \subseteq N_G(H)$; that is, $H$ is normalized by $[G,H]$ if and only if it is normalized by $\overline{[G,H]}$. Since $\overline{H^G} = \overline{H \cdot [H,G]} = \overline{H \cdot [G,H]}$, we conclude that $\overline{H^G}$ normalizes $H$ if and only if $\overline{[G,H]}$ normalizes $H$. \end{remark}

We end this section with some basic results used throughout the article, with proofs included for convenience. 

Let $G$ be a residually finite group; we identify $G$ with its image in $\widehat{G}$, so that $G$ is a dense subgroup of $\widehat{G}$. A subgroup $H\le G$ is said to be closed in the profinite topology of $G$ if $H=\overline{H}\cap G$, where $\overline{H}$ denotes the closure of $H$ in $\widehat{G}$; equivalently, if $H$ is the intersection of all finite-index subgroups of $G$ containing $H$. The group $G$ is \emph{LERF} (Locally Extended Residually Finite) if every finitely generated subgroup of $G$ is closed in the profinite topology of $G$. If the same holds for every subgroup of $G$ (not necessarily finitely generated), $G$ is said to be \emph{ERF} (Extended Residually Finite). By a classical theorem of Mal'cev (see \cite[Theorem~5.4.16]{robinson}), every polycyclic group is ERF, and consequently LERF.

\begin{lemma} \label{fechos normais}
Let $H$ be a subgroup of a residually finite group $G$ and let $\widehat{G}$ be its profinite completion. Then
\[
\overline{\overline{H}^{\widehat{G}}} = \overline{H^G}.
\]
\end{lemma}
\begin{proof}
Since $H \le H^{G}$, we have $\overline{H} \le \overline{H^{G}}$ (closures taken in $\widehat{G}$). Moreover, $H^{G} \unlhd G$, so $G$ normalizes $H^{G}$; as the conjugation by elements of $G$ is continuous on $\widehat{G}$, $G$ normalizes $\overline{H^{G}}$ as well. Since $G$ is dense in $\widehat{G}$, it follows that $\overline{H^{G}} \unlhd_c \widehat{G}$: indeed, fix $y \in \overline{H^G}$. The map $\varphi_y(x) = x^{-1}yx$ $(x \in \widehat{G})$ is continuous on $\widehat{G}$, and $\varphi_y(G) \subseteq \overline{H^G}$. By continuity, $\varphi_y(\widehat{G}) = \varphi_y(\overline{G}) \subseteq \overline{\varphi_y(G)} \subseteq \overline{H^G}$, since $\overline{H^G}$ is closed. Hence $x^{-1}yx \in \overline{H^G}$ for all $x \in \widehat{G}$, i.e.\ $\overline{H^G} \unlhd_c \widehat{G}$. Since $H \le \overline{H} \le  \overline{H^G}$ and $\overline{H^G} \unlhd_c \widehat{G}$, we get $\overline{\overline{H}^{\widehat{G}}} \le \overline{H^G}$. 

On the other hand, $H^G$ is generated by the conjugates $g^{-1}hg$, for $g \in G$ and $h \in H$, all of which lie in the closed normal subgroup $\overline{\overline{H}^{\widehat{G}}}$, so $H^G \le \overline{\overline{H}^{\widehat{G}}}$. Since $\overline{\overline{H}^{\widehat{G}}}$ is closed, and $\overline{H^G}$ is the intersection of all closed sets containing $H^G$, it follows that $\overline{H^G} \le \overline{\overline{H}^{\widehat{G}}}$. Hence the equality holds.
\end{proof}

\begin{lemma}\label{lem_polycyclic_normal_closure}
Let $G$ be a group with the property \textnormal{ERF} (e.g., a polycyclic group), and let $H$ be a subgroup of $G$. Then \[
\overline{\overline{H}^{\widehat{G}}} \cap G = H^G = \overline{H}^{\widehat{G}} \cap G.
\]
\end{lemma}
\begin{proof}
Since $H^G \le G$ and $G$ is ERF, $H^G$ is closed in the profinite topology of $G$, so $\overline{H^G} \cap G = H^G$. By Lemma \ref{fechos normais}, $\overline{H^G} = \overline{\overline{H}^{\widehat{G}}}$, hence $\overline{\overline{H}^{\widehat{G}}} \cap G = H^G$. Since $\overline{H}^{\widehat{G}} \subseteq \overline{\overline{H}^{\widehat{G}}}$, we obtain $\overline{H}^{\widehat{G}} \cap G \subseteq H^G$. The reverse inclusion is immediate. Therefore $\overline{H}^{\widehat{G}} \cap G = H^G$.
\end{proof}

Let $G$ be a profinite group. 
The upper central
series $(Z_i(G))_{i \geq 0}$ and the lower central series
$(\gamma_n(G))_{n \geq 1}$ of $G$ are defined, respectively, by
$$Z_0(G) = \{1\}, \quad Z_{i+1}(G)/Z_i(G) = Z(G/Z_i(G))
\quad (i \geq 0),$$
$$\gamma_1(G) = G, \quad \gamma_{n+1}(G) = \overline{[\gamma_n(G), G]}
\quad (n \geq 1),$$
where the bar denotes the closure in $G$. We say that $G$ is nilpotent
of class at most $r$ if $\gamma_{r+1}(G) = \{1\}$, equivalently, if
$Z_r(G) = G$. We denote by \(\gamma_n^{abs}(G)\) 
the abstract a nalogue, defined by \(\gamma_1^{abs}(G) = G\) and 
\(\gamma_{n+1}^{abs}(G) = [\gamma_n^{abs}(G), G]\) without taking closures, so 
that \(\gamma_n^{abs}(G) \leq \gamma_n(G)\). We denote by \(\widehat{\mathfrak{N}}_r\) the class of nilpotent profinite groups of class at most \(r\), and by \(\mathfrak{N}_r\) its abstract analogue. 

The next proposition is a well-known result about subnormality, the abstract version of which can be found in Hall \cite[p.~154]{hall 2}. Although the proof for the profinite case is analogous, we include it here for completeness, as a direct reference seems unavailable.

\begin{proposition} \label{defecttwo} 
If $G$ is a nilpotent profinite group of class $k$, then every closed subgroup of $G$ has subnormal defect at most $k$. 
\end{proposition} 

\begin{proof} 
Let $H \leqslant_c G$ and consider the upper central series of closed subgroups of $G$:  
\[ \{1\}=Z_{0}(G)\unlhd_{c} Z_{1}(G)=Z(G) \unlhd_{c} \cdots \unlhd_{c} Z_{k-1}(G) \unlhd_{c} Z_{k}(G) = G, \]
where $Z_i(G)/Z_{i-1}(G) = Z(G/Z_{i-1}(G))$ for each $i=1, 2, \ldots, k$.

For any $j \in \{0, 1, \ldots, k-1\}$, the definition of the upper
central series gives $[Z_{j+1}(G), G] \leq Z_j(G)$, and since
$Z_j(G)$ is closed, also $\overline{[Z_{j+1}(G), G]} \leqslant_c
Z_j(G)$. Consequently, $Z_{j+1}(G)$ normalizes the product $HZ_j(G)$.
Since $H$ trivially normalizes $HZ_j(G)$, it follows that
$HZ_{j+1}(G)$ normalizes $HZ_j(G)$. Furthermore, since the product of a
closed subgroup and a closed normal subgroup in a profinite group is
closed, each $HZ_j(G)$ is a closed subgroup. Thus, we obtain the finite chain of closed subgroups:
\[ H = HZ_{0}(G) \unlhd_{c} HZ_{1}(G) \unlhd_{c} \cdots \unlhd_{c} HZ_{k-1}(G) \unlhd_{c} HZ_{k}(G) = G. \]
As a consequence, $H$ is subnormal in $G$ in at most $k$ steps, which means that the subnormal defect of $H$ in $G$ is at most $k$.
\end{proof}

\section{Some classes of profinite groups}

In this section, we introduce specific classes of profinite groups that generalize the abstract group classes investigated by Mahdavianary \cite{mahdavimain}. Thereafter, we prove that non-periodic profinite groups without even-order elements in which all procyclic subgroups are subnormal of defect at most 2 are precisely profinite $2$-Engel groups. This establishes a profinite generalization of Heineken's result \cite[Theorem~1(I)]{heineken}.

For a profinite group $G$, we define some classes as follows:
\begin{enumerate}
    \item[(i)] $G \in \widehat{\mathscr{T}}$ if every procyclic subgroup of $G$ is $2$-subnormal.
    \item[(ii)] $G \in\widehat{\mathscr{A}}$ if every closed abelian subgroup of $G$ is $2$-subnormal.
    \item[(iii)] $G \in\widehat{\mathscr{S}}$ if every closed subgroup of $G$ is $2$-subnormal.
\end{enumerate}

These definitions naturally yield the inclusions $\widehat{\mathscr{S}} \subseteq \widehat{\mathscr{A}} \subseteq \widehat{\mathscr{T}}$. Similarly, Mahdavianary \cite{mahdavimain} defined the corresponding classes $\mathscr{S}, \mathscr{A}, \mathscr{T}$ for abstract groups, focusing on all subgroups, abelian subgroups, and cyclic subgroups, respectively, having subnormal defect at most $2$. Note that if $G$ is a finite group equipped with the discrete topology, then $G \in \widehat{\mathscr{S}}$ if and only if $G \in \mathscr{S}$. The analogous equivalences hold for the classes $\widehat{\mathscr{A}}$ and $\widehat{\mathscr{T}}$.

\bigskip

The abstract case of the following definition was introduced by Cappitt \cite{Cap71}. 

\begin{definition}
A profinite Cappitt group is a profinite group $G$ such that the closed subgroup \(\tilde{S}(G) = \overline{ \langle L \leqslant_c G \mid L \ntriangleleft G \rangle}\) is a proper subgroup of $G$. 
\end{definition}

A direct application of Proposition \ref{defecttwo} yields the following corollary. It shows that (profinite) Cappitt groups belong to $\mathscr{S}$
(respectively, $\widehat{\mathscr{S}}$). 

\begin{corollary}
If $G$ is a Cappitt group (respectively, a profinite Cappitt group), then every subgroup (respectively, closed subgroup) of $G$ has subnormal defect at most $2$.  
\end{corollary}

\begin{proof}
If $G$ is a Cappitt group or a profinite Cappitt group, then $G$ is known to be nilpotent of class at most $2$ (see \cite{Por19} and \cite{Cap71}). Therefore, by Proposition \ref{defecttwo}, every (closed) subgroup of $G$ has subnormal defect at most $2$.
\end{proof}

\bigskip

\subsection{Profinite groups in $\widehat{\mathscr{T}}$ and $\widehat{\mathscr{A}}$} 

\bigskip 

We begin with a series of auxiliary lemmas that govern the structural behavior of these classes.

\begin{lemma} \label{formula0}
If $G \in \widehat{\mathscr{T}}$, then $G/U \in \mathscr{T}$ for every open normal subgroup $U \unlhd_{o} G$.
\end{lemma}
\begin{proof}
Suppose $U \unlhd_{o} G$ and let $L_U = \langle gU \rangle \le G/U$, where $g \in G$. Since $G \in \widehat{\mathscr{T}}$, we have $\overline{\langle g \rangle} \unlhd_c \overline{\overline{\langle g \rangle}^{\,G}}$, hence $\overline{\langle g \rangle}\,U \unlhd_c \overline{\overline{\langle g \rangle}^{\,G}}\, U$. Since $\pi_U \colon G \to G/U$ is continuous and $G/U$ is finite with the discrete topology, every subset of $G/U$ is closed, so $\pi_U$ commutes with closures. Therefore

\[
L_U = \frac{\overline{\langle g \rangle}\,U}{U}
= \langle gU \rangle
\unlhd
\frac{\overline{\overline{\langle g \rangle}^{\,G}}\,U}{U}
= \pi_U\!\left(\overline{\overline{\langle g \rangle}^{\,G}}\right)
= \pi_U\!\left(\overline{\langle g \rangle}^{\,G}\right)^{\,G/U}
= \langle gU \rangle^{\,G/U}
= L_U^{\,G/U}.
\] 

This proves that every cyclic subgroup of $G/U$ is $2$-subnormal; hence $G/U \in \mathscr{T}$.
\end{proof}

\begin{lemma} \label{class3}
If $G \in \widehat{\mathscr{T}}$, then $G \in \widehat{\mathfrak{N}}_{3}$. In particular, every group in $\widehat{\mathscr{T}}$ is a metabelian profinite group.
\end{lemma}
\begin{proof}
Let $\mathcal{N} = \{U \mid U \unlhd_{o} G\}$ be the fundamental system of open normal neighborhoods of the identity in $G$. It is a standard result \cite[Theorem~2.1.3]{RZ10} that $G = \varprojlim_{U \in \mathcal{N}} G/U$. By \cite[Proposition~1.1.5]{wilson98}, $G$ embeds as a closed subgroup into the Cartesian product $\prod_{U \in \mathcal{N}} G/U$.  Since $G \in \widehat{\mathscr{T}}$, Lemma \ref{formula0} ensures that $G/U \in \mathscr{T}$ for all $U \in \mathcal{N}$. According to Garrison and Kappe \cite[Theorem~1.3]{kappegar} (or Mahdavianary \cite[Theorem~1]{mahdavinilpotente}), any abstract group in $\mathscr{T}$ belongs to $\mathfrak{N}_{3}$. Thus, $G/U \in \mathfrak{N}_{3}$ for all $U \in \mathcal{N}$. Because the class of nilpotent groups of class at most 3 is closed under Cartesian products and closed subgroups (see Proposition~4.2 of \cite{LIM}), we conclude that $G \in \widehat{\mathfrak{N}}_{3}$. Furthermore, commutator identities imply that $G'' = \overline{[\gamma_2(G), \gamma_2(G)]} \leqslant \gamma_4(G) = \{1\}$, which confirms that $G$ is metabelian.
\end{proof}

\begin{corollary} \label{cor4} 
The following inclusions hold: $\widehat{\mathscr{S}}  \subseteq \widehat{\mathscr{A}} \subseteq \widehat{\mathfrak{N}}_{3}.$
\end{corollary}
\begin{proof} 
By definition, $\widehat{\mathscr{S}}  \subseteq \widehat{\mathscr{A}} \subseteq \widehat{\mathscr{T}}$. By Lemma \ref{class3}, $\widehat{\mathscr{T}} \subseteq \widehat{\mathfrak{N}}_{3}$. The result follows immediately.
\end{proof}

\begin{corollary} \label{G3engel}
If $G \in \widehat{\mathscr{T}}$, then $G$ is a profinite $3$-Engel group.
\end{corollary}
\begin{proof} 
By Lemma \ref{class3}, $\gamma_4(G) = 1$. Thus every commutator of weight $4$ is trivial; 
in particular, $[v, _3 u] = [v, u, u, u] = 1$ for all $u, v \in G$. 
\end{proof}

The identity in the following lemma corresponds to Heineken \cite[Lemma~1]{heineken2}, adapted to standard modern notation.

\begin{lemma} \label{lemmarefalemao} 
If $G$ is a profinite $3$-Engel group, then $[a, [b,a], [b,a]] = 1$ for all $a,b \in G.$
\end{lemma}
\begin{proof} 
First, we expand the commutator as follows:
\begin{equation} \label{eq 3.6}
    [a,[b,a],[b,a]] = [a,[b,a], (a^{-1})^b a] = [a,[b,a],a] \cdot [a,[b,a],(a^{-1})^b]^a. 
\end{equation} 
Since $G$ is $3$-Engel, $[b,_3 a] = [[[b,a],a],a] = 1$. Utilizing Lemma \ref{prop comutadores}\eqref{comut inverso}, we find:
\begin{equation} \label{usar no fim} 
    [b, a, a, a] = 1 \Longleftrightarrow [[a,[b,a]],a] = 1 \Longleftrightarrow [[a,[b,a]],a^{-1}] = 1.
\end{equation}
Therefore, the first term on the right-hand side of \eqref{eq 3.6} vanishes, and it suffices to prove that $[a,[b,a],(a^{-1})^b]^a = 1$.

Let $x, y \in G$. The $3$-Engel condition ensures that
\[ 1 = [[[y^{-1},x],x],x] = [[(x^{-1})^{y^{-1}}x,x],x] = [[(x^{-1})^{y^{-1}},x]^x [x,x],x] = [[(x^{-1})^{y^{-1}},x]^x,x]. \] 
This simplifies to:
\begin{equation} \label{eq 3.8} 
    [[[(x^{-1})^{y^{-1}},x],x]]^x=1 \Longleftrightarrow [[[(x^{-1})^{y^{-1}},x],x]]=1.
\end{equation} 
By substituting $x = (a^{-1})^{b}$ and $y=b$ into \eqref{eq 3.8} (so that \((x^{-1})^{y^{-1}} = a\)), we obtain the identity:
\[ 1 = [a, (a^{-1})^b, (a^{-1})^b]. \]
Expanding this applying Lemma \ref{prop comutadores} \eqref{comut conjugado} and \eqref{comut produto dir}, we arrive at:
\begin{align} \label{eq 3.9}
    1 &= [a, (a^{-1})^b, (a^{-1})^b] = [[a, [b,a]a^{-1}], (a^{-1})^b] \nonumber \\ 
      &= [[a,a^{-1}] \cdot [a,[b,a]] \cdot [a,[b,a], a^{-1}], (a^{-1})^b] \nonumber \\
      &= [[a,[b,a]] \cdot [a,[b,a], a^{-1}], (a^{-1})^b].
\end{align}
Substituting the result of \eqref{usar no fim} into \eqref{eq 3.9} yields
\[ 1 = [a, [b,a], (a^{-1})^b], \]
which completes the proof.
\end{proof}

The following result is the profinite adaptation of Lemma 1 in \cite{heineken}.

\begin{lemma} \label{lemma1}
Let $G$ be a profinite group. Then $G \in \widehat{\mathscr{T}}$ if and only if $[v,u,u] \in \overline{\langle u \rangle}$ for all $u, v \in G$.
\end{lemma}
\begin{proof}
For any $u,v \in G$, we have the basic identity
\begin{equation} \label{formuladomahdavi}
    u^{u^v} = u\,[u,[u,v]]
    \;\Longleftrightarrow\;
    (u^{u^v})^{-1} u = [u,[u,v]]^{-1}=[u,v , u].
\end{equation}

Suppose $G \in \widehat{\mathscr{T}}$, meaning that
$\overline{\langle u \rangle} \unlhd_c \overline{\overline{\langle u \rangle}^{\,G}}$
for all $u \in G.$ Utilizing \eqref{formuladomahdavi} and Lemma \ref{prop comutadores}\eqref{comut tirar inverso}, we obtain
\[
  ([v,u,u]^{-1})^{[u,v]}
  = [u,[v,u]]^{[u,v]}
  = [u,[v,u]]^{[v,u]^{-1}}
  = [u,v,u]
  = (u^{u^v})^{-1} u.
\]
Since $u^{u^g} \in \overline{\langle u \rangle}$ for every $g \in G$, it follows that
\[
  ([v,u,u]^{-1})^{[u,v]}
  = [u,v,u]
  = (u^{u^v})^{-1} u
  \in \overline{\langle u \rangle}.
\]
Every conjugate \(u^g\) of \(u\) lies in \(N_G(\overline{\langle u \rangle})\). Since these conjugates generate \(\langle u \rangle^G\) and \(N_G(\overline{\langle u \rangle})\) is a subgroup, \(\langle u \rangle^G \subseteq N_G(\overline{\langle u \rangle})\). As \(N_G(\overline{\langle u \rangle})\) is closed, it contains \(\overline{\langle u \rangle^G}\). From the fact that \(\overline{\langle u \rangle}^G \subseteq \overline{\langle u \rangle^G}\), we get \(\overline{\langle u \rangle}^G \subseteq N_G(\overline{\langle u \rangle})\). Given that \([v,u]=(u^{-1})^v u\in\overline{\langle u\rangle}^G\), conjugation by \([v,u]\) preserves \(\overline{\langle u\rangle}\). Conjugating \(([v,u,u]^{-1})^{[u,v]}\in\overline{\langle u\rangle}\) by \([v,u]=[u,v]^{-1}\) gives \([v,u,u]^{-1}\in\overline{\langle u\rangle}\), hence \([v,u,u]\in\overline{\langle u\rangle}\). 

Conversely, assume that $[v,u,u] \in \overline{\langle u \rangle}$ for all $u, v \in G$. Applying commutator identities (see Lemma \ref{prop comutadores}), we derive:
\begin{align} \label{eqparadesconjugar}
    ([u,v,u]^{-1})^{u^{-2}} &= [u,[u,v]]^{u^{-2}} = \left([u,[u,v]]^{u^{-1}}\right)^{u^{-1}} = [[u,v],u^{-1}]^{u^{-1}} \nonumber \\ 
    &= [[u,v]^{u^{-1}},(u^{-1})^{u^{-1}}] = [[v,u^{-1}],u^{-1}] \nonumber \\
    &= [v, u^{-1}, u^{-1}] \in \overline{\langle u^{-1} \rangle} = \overline{\langle u \rangle}.
\end{align}
Conjugating \eqref{eqparadesconjugar} by $u^2$ yields
\[ [u,[u,v]] = [u,v,u]^{-1} = [v, u^{-1}, u^{-1}]^{u^2} \in \overline{\langle u \rangle}. \]

Finally, observing that $u \cdot [u,[u,v]] = u^{u^v}$ for every $v \in G$, we find that $u^{u^v} \in \overline{\langle u \rangle}$. Hence every conjugate of $u$ lies in $N_G(\overline{\langle u \rangle})$, and since these conjugates generate \(\langle u \rangle^G\), it follows that \(\overline{\langle u \rangle} \trianglelefteq \langle u \rangle^G\). Since \(N_G(\overline{\langle u \rangle})\) is closed, we get \(\overline{\langle u \rangle^G} \subseteq N_G(\overline{\langle u \rangle})\), i.e.\ \(\overline{\langle u \rangle} \trianglelefteq \overline{\langle u \rangle^G}\). As \(\overline{\langle u \rangle^G} = \overline{\overline{\langle u \rangle}^{\,G}}\), we obtain \(\overline{\langle u \rangle} \trianglelefteq_c \overline{\overline{\langle u \rangle}^{\,G}}\), so \(G \in \widehat{\mathscr{T}}\).
\end{proof}

\begin{remark} \label{zeta}
Let \(G\) be a profinite group and \(a\in G\). Since $\widehat{\mathbb{Z}}$ is the free profinite group on $\{1\}$, there is a unique continuous epimorphism $\zeta\colon\widehat{\mathbb{Z}}\to\overline{\langle a\rangle}$ with $\zeta(1)=a$ (see \cite[Section~4.1]{RZ10}). Since $\zeta$ is surjective, every element of $\overline{\langle a\rangle}$ is of the form $a^{\lambda}=\zeta(\lambda)$ for some $\lambda\in\widehat{\mathbb{Z}}$. 
\end{remark}

The following lemma adapts Heineken \cite[Lemma~2]{heineken2} to the profinite framework.

\begin{lemma} \label{olambda}
If $G \in \widehat{\mathscr{T}}$, then $[b,a,a^{\lambda}]=[b,a,a]^{\lambda}$ for all $a,b \in G$ and all $\lambda \in \widehat{\mathbb{Z}}$.    
\end{lemma}
\begin{proof} 
By Corollary \ref{G3engel}, $G$ is a $3$-Engel profinite group. If $[b,a,a]=1$, the result is trivial since $[b,a] \in C_G\left(\overline{\langle a \rangle}\right).$ Assume hereafter that $[b,a,a] \neq 1$. Since $[b,a,a] \in \overline{\langle a \rangle}$ by Lemma \ref{lemma1}, we may write $[b,a,a] = a^{\mu}$ for some $\mu \in \widehat{\mathbb{Z}}$ not in $\ker(\zeta)$. 

By induction on $n \in \mathbb{N}$, one can verify that $[b,a,a^{n}]=[b,a, a]^{n}$ for all $a,b \in G$. For a negative integer $n=-m$ ($m>0$), we observe that since $[b,a,a]$ belongs to $C_{G}(\overline{\langle a \rangle})$,
\begin{align*}
    [b,a,a^{-m}] &= a^m \cdot a^{-m} \cdot \left( [b,a]^{-1} a^m [b,a] a^{-m}\right) = a^{m} \left([b,a,a^m]\right)^{-1} a^{-m} \\
    &= a^{m} [b,a,a]^{-m} a^{-m} = \left( a [b,a,a] a^{-1} \right)^{-m} = [b,a,a]^{-m}.
\end{align*}
Thus, the property holds for all $n \in \mathbb{Z}$. 

Take $\lambda \in \widehat{\mathbb{Z}}$. Since $\widehat{\mathbb{Z}} = \overline{\mathbb{Z}}$ is metrizable, take a sequence of integers $(n_i)_{i\in\mathbb{N}}$ converging to $\lambda$ in $\widehat{\mathbb{Z}}$. Since $\zeta$ is continuous, $a^{n_i}\to a^{\lambda}$ (see \cite[Lemma~4.1.1]{RZ10}) in \(\overline{\langle a \rangle}\). Consequently, for every \(i\),
\begin{equation} \label{eq:pot}
[b,a,a]^{n_i} = (a^{\mu})^{n_i} = a^{\mu n_i}.
\end{equation}
Since $n_i\to \lambda$ in $\widehat{\mathbb{Z}}$ and multiplication in $\widehat{\mathbb{Z}}$ is continuous, we have $\mu n_i \to \mu\lambda$ in $\widehat{\mathbb{Z}}$. Applying Lemma~4.1.1(a) to the element $a$, we get
\begin{equation} \label{pr lim}
[b,a,a]^{n_i} = a^{\mu n_i} \to a^{\mu\lambda} = (a^{\mu})^{\lambda} = [b,a,a]^{\lambda} \quad \text{in } \overline{\langle a \rangle}.
\end{equation}

On the other hand, from $[b,a,a] = a^{\mu}$ we get \(\left(a^{-1}\right)^{[b,a]} = a^{\mu-1} \in \overline{\langle a \rangle}\), so \([b,a]\) normalizes \(\overline{\langle a \rangle}\). Hence the inner automorphism \(\tau_{[b,a]}\) of \(G\), defined by
\[
\tau_{[b,a]}(x) = [b,a]^{-1} x [b,a], \quad x \in G,
\]
restricts to a continuous automorphism of \(\overline{\langle a \rangle}\). Since \(\tau_{[b,a]}\) is continuous on \(\overline{\langle a \rangle}\), from \(a^{n_i} \to a^\lambda\) we get
\[
\tau_{[b,a]}(a^{n_i}) \to \tau_{[b,a]}(a^\lambda) \quad \text{in } \overline{\langle a \rangle}.
\]
Using the continuity of inversion and multiplication in \(\overline{\langle a \rangle}\), it follows that
\begin{equation} \label{seg lim}
[b,a,a^{n_i}] = [b,a]^{-1}(a^{n_i})^{-1}[b,a]\,a^{n_i} \to [b,a]^{-1}(a^{\lambda})^{-1}[b,a]\,a^{\lambda} = [b,a,a^{\lambda}] \quad \text{in } \overline{\langle a \rangle}.
\end{equation}

Since $[b,a,a]^{n_i} = [b,a,a^{n_i}]$ for all $i$, and both sides converge in \(\overline{\langle a \rangle}\), which is Hausdorff, the limits in \eqref{pr lim} and \eqref{seg lim} must coincide:
\[
[b,a,a]^{\lambda} = [b,a,a^{\lambda}].
\]
\end{proof}

The next lemma outlines essential identities for right \(2\)-Engel elements. A detailed proof is available in Robinson \cite[Theorem~7.13]{chapter7}.

\begin{lemma} \label{identitiesof2engels} 
Let $w$ be a right $2$-Engel element, and let $x, y, z$ be arbitrary elements of a group $G$. Then:
\begin{enumerate}
    \item [\textnormal{(i)}] $w$ is a left $2$-Engel element.
    \item [\textnormal{(ii)}] $[w,x,y]=[w,y,x]^{-1}.$
    \item [\textnormal{(iii)}] $[w, [x,y]]=[w,x,y]^{2}.$
    \item [\textnormal{(iv)}] $[w, [x,y], z]=1.$    
\end{enumerate}
\end{lemma}

We shall use the following identity, due to Kappe \cite[Corollary~2.2]{kappe}, to prove that the right-hand \(2\)-Engel element set \(R(G)\) is a closed characteristic subgroup of a profinite group \(G\).  

\begin{lemma} \label{produtodisengel e doisengel}
If $x, y$ are right $2$-Engel elements of a group $G$, then $[yx, g, g] = 1$ and $[g,[g,yx]] = 1$ for all $g \in G.$ 
\end{lemma}
\begin{proof} 
Expanding the commutator with the identity in Lemma \ref{prop comutadores}\eqref{comut produto esq} and applying the property that $x$ is a right $2$-Engel element yields:
\[ [yx, g , g]=[[yx, g], g] = [[y,g]^x [x,g], g] = [[y,g]^x, g]^{[x,g]} \cdot [[x,g],g] = [[y,g]^x, g]^{[x,g]}, \, \forall g \in G. \]
Further expansion of the right-hand term provides:
\[ [[y,g]^x, g]^{[x,g]} = [[y,g], g^{x^{-1}}]^{x[x,g]} = [[y,g], g[g, x^{-1}]]^{x[x,g]} = [[y,g], [g, x^{-1}]]^{x[x,g]}, \, \forall g \in G. \]
Applying Lemma \ref{identitiesof2engels} (ii) and (iv), along with the fact that $y$ is a right $2$-Engel element, we obtain:
\[ [yx, g, g]= [[yx, g], g] = [[y,g], [g, x^{-1}]]^{x[x,g]} = \left([y, [g, x^{-1}], g]^{-1}\right)^{x[x,g]} = 1, \forall g \in G. \]
Since $[yx, g]$ commutes with $g$, it is immediate that $[g, [g, yx]] = 1$.
\end{proof}

\begin{proposition} \label{caracteristico}
The set $R(G)$ of right $2$-Engel elements of a profinite group $G$ forms a closed characteristic subgroup. 
\end{proposition}

\begin{proof} 
We first prove that $R(G)$ is closed. For each $g\in G$, the map \(\rho_g\colon G\to G\) given by \(\rho_g(x)=[x,g,g]\) is continuous, and \(R(G)=\bigcap_{g\in G}\rho_g^{-1}(\{1\})\). Hence \(R(G)\) is closed in \(G\), being an intersection of closed subsets. Also, \(1 \in R(G) \neq \emptyset\). By Lemma \ref{produtodisengel e doisengel}, 
\(R(G)\) is closed under the product. Thus, it remains to show that if 
\(a \in R(G)\), then \(a^{-1} \in R(G)\). Indeed, since $a \in R(G)$, Lemma \ref{identitiesof2engels} (i) shows that $a$ is a left $2$-Engel element. Hence, for every $x \in G$, the centralizer $C_G(\langle [x,a]\rangle)$ contains $a$, $a^{-1}$ and $x$. Therefore it follows from Lemma \ref{prop comutadores} (6) that  
$$[a^{-1},x,x] = [[x,a]^{a^{-1}}, x]  = [x,a,x] = 1.$$ Finally, let \(\varphi\) be a continuous automorphism of \(G.\) If \(a \in R(G),\) then \(\varphi([a,g,g])=[\varphi(a), \varphi(g), \varphi(g)]=1,\) whence it follows that \(\varphi(a) \in R(G),\) therefore \(R(G)\) is characteristic.
\end{proof}

The following lemma  characterizes the groups in class \(\widehat{\mathscr{A}}\), adapting \cite[Corollary~3(ii)]{mahdavimain} to profinite groups.

\begin{lemma} \label{2geradob} 
Let $G$ be a profinite group. Then $G \in \widehat{\mathscr{A}}$ if and only if $[v,x,y] \in \overline{\langle x,y \rangle}$ for all $v, x, y \in G$ satisfying $[x,y]=1$.   
\end{lemma}
\begin{proof} 
Assume $G \in \widehat{\mathscr{A}}$. For $x, y \in G$ with $[x,y]=1$, $\overline{\langle x,y \rangle}$ is abelian and satisfies $\overline{\langle x,y \rangle} \unlhd_c \overline{\overline{\langle x,y \rangle}^G}$. For any $v \in G$, we analyze the conjugate $y^{[v,x]}$. Notice that $y^{[v,x]} = y^{(x^{-1})^{v} x}$. Moreover, \((x^{-1})^v x \in \overline{\langle x \rangle}^G \leq \overline{\langle x, y \rangle}^G\). Hence \((x^{-1})^v x \in \overline{\overline{\langle x \rangle}^G} \leq_c \overline{\overline{\langle x, y \rangle}^G}\). Therefore, the normality condition forces \(y^{[v,x]} \in \overline{\langle x, y \rangle}\).  
Expanding the conjugate yields:
\[ y^{[v,x]} = y[y, [v,x]] = y[[v, x], y]^{-1} = y[v,x,y]^{-1}. \]
Therefore, $[v,x,y] = (y^{[v,x]})^{-1} y \in \overline{\langle x,y \rangle}$.

Conversely, assume $[v,x,y] \in \overline{\langle x,y \rangle}$ for all $x,y$ commuting in $G$. Let $H$ be an arbitrary closed abelian subgroup of $G$. We want to prove that $H = H_2 = \overline{H^{\overline{H^G}}}$, i.e.\ $H \unlhd_c \overline{H^G} \unlhd_c G$, so $H$ has subnormal defect at most $2$. By the equivalence in Remark \ref{aremark}, for any $x, y \in H$ and $v \in G$, we must show that $y^{[v,x]} \in H$ and $y^{[x,v]} \in H$ to guarantee that $H$ is normalized by \(\overline{[G,H]}\), and consequently by \(\overline{H^G}\). First, for any $x, y \in H$ and $v \in G$, we have
$y^{[v,x]} = y[v,x,y]^{-1} \in \overline{\langle x,y \rangle} \le H$,
by the hypothesis and since $H$ is closed. Replacing $x$ by $x^{-1}$, we get
$y^{[v,x^{-1}]} = y[v,x^{-1},y]^{-1} \in \overline{\langle x^{-1},y \rangle} = \overline{\langle x,y \rangle} \le H$.
It remains to compute
$y^{[x,v]} = y^{[v,x^{-1}]^x} = y^{x^{-1}[v,x^{-1}]x} = \left((y^{x^{-1}})^{[v,x^{-1}]}\right)^x = \left(y^{[v,x^{-1}]}\right)^x \in \overline{\langle x,y \rangle}^x = \overline{\langle x,y \rangle} \le H$,
where the first equality follows from Lemma \ref{prop comutadores} (6).
Therefore $H = H_2$, and hence $G \in \widehat{\mathscr{A}}$.
\end{proof}


\section{Profinite groups with procyclic subgroups whose subnormal defect is at most 2}

The first result below (Proposition~\ref{G2engeldemais}) is the profinite
analogue of Theorem~1(1) of \cite{heineken}; it is the key ingredient for the
rest of the article. In the abstract case the torsion-freeness of an infinite
cyclic subgroup is automatic, since it is isomorphic to \(\mathbb{Z}\); in the
profinite setting it is not, and the relevant restriction is therefore to all
infinite procyclic subgroups be torsion-free.

\begin{proposition}[\textnormal{Profinite version of Theorem~1(1) in \cite{heineken}}]  
\label{G2engeldemais}
Let \(G \in \widehat{\mathscr{T}}\) be a non-periodic profinite group all of
whose infinite procyclic subgroups are torsion-free. Then \(G\) is a
\(2\)-Engel group. 
\end{proposition}

\begin{proof} 
The proof will be divided into three distinct steps.

\textbf{Step 1:} We first show that $[g,u,u] = 1$ for any element $u \in G$ of infinite order and any $g \in G$.

Assume for a contradiction that $[g,u,u] \neq 1$. Let \(\zeta\colon \widehat{\mathbb{Z}} \to \overline{\langle u \rangle}\) be the continuous epimorphism defined by \(\alpha \mapsto u^\alpha\) (see Remark \ref{zeta}). By Lemma \ref{lemma1}, $[g,u,u] \in \overline{\langle u \rangle}$, so we may write $[g,u,u] = u^{\lambda}$ for some $\lambda \in \widehat{\mathbb{Z}}$. Suppose, for contradiction, that \(\lambda \notin \ker(\zeta)\). Since \(u\) has infinite order, the hypothesis gives
\(\overline{\langle u \rangle} \cong \prod_{p \in \pi} \mathbb{Z}_p\),
with \(\pi\) a non-empty set of primes. 

By Corollary \ref{G3engel}, $G$ is a $3$-Engel profinite group. Applying Lemma \ref{lemmarefalemao}, we have
\[ 1 = [u, [g,u], [g,u]]. \]
Applying the commutator identity \eqref{comut tirar inverso} from Lemma \ref{prop comutadores}, we obtain:
\[ 1 = \left[ [g,u,u]^{-1}, [g,u] \right] = \left[ (u^{\lambda})^{-1}, [g,u] \right] = [g,u, u^{\lambda}]^{(u^{\lambda})^{-1}}. \]
Conjugating this equation by $u^{\lambda}$ yields \([g,u,u^{\lambda}] = 1\). By Lemma \ref{olambda}, we know that \([g,u,u^{\lambda}] = [g,u,u]^{\lambda}\). Therefore,
\[ 1 = [g,u,u^{\lambda}] = [g,u,u]^{\lambda} = (u^{\lambda})^{\lambda} = u^{\lambda^2} \Longrightarrow \lambda ^2 \in \operatorname{ker}(\zeta).\]

Since \(\lambda \notin \ker(\zeta)\), the class \(\lambda + \ker(\zeta)\) is nonzero. But \((\lambda + \ker(\zeta))^2 = 0\), so this class is nilpotent. This is impossible, since \(\widehat{\mathbb{Z}}/\ker(\zeta) \cong \prod_{p \in \pi'} \mathbb{Z}_p\) has no nonzero nilpotent elements. Hence \(\lambda \in \ker(\zeta)\), so \(u^\lambda = 1\). Since \([g,u,u] = u^\lambda\), we conclude that
\begin{equation*}
    [g,u,u] = 1, \text{ for any } u \in G \text{ of infinite order and } g \in G.
\end{equation*}

\textbf{Step 2:} We now prove that $[u,g,g] = 1$ for any infinite-order element $u \in G$ and any $g \in G$.

If $g$ has infinite order, the result follows immediately from Step 1 by reversing the roles of $u$ and $g$. Hence, assume $g$ has finite order. Let $W = \langle u, g \rangle$ be the abstract subgroup of $G$ generated by $u$ and $g$. 
Since $G$ is a nilpotent profinite group of class at most $3$ by Lemma \ref{class3}, 
we have $\gamma_4(G) = \{1\}$. As $W \leq G$, it follows that 
$\gamma_4^{abs}(W) \leq \gamma_4(G) = \{1\}$, so $W$ is a nilpotent abstract group of class at most $3$. 
Since $W$ is finitely generated and nilpotent, the set of elements of finite order 
forms a subgroup \(\operatorname{tor}(W)\) of \(W\) (see Kargapolov and Merzljakov 
\cite[Lemma~16.2.1, Theorem~16.2.7]{kargapolov}). Since $u$ has infinite order and $g$ has finite order, the product $gu$ must have infinite order (otherwise, $u = g^{-1}(gu)$ would have finite order, a contradiction). 

Applying the result of Step 1 to the elements $u$ and $gu$, both of infinite order, 
and utilizing the commutator identity \eqref{comut produto dir} from Lemma 
\ref{prop comutadores}, we find: 
\[ 1 = [g, gu, gu] = [[g,u], gu] = [[g,u], u] \cdot [[g,u], g]^u = 1 \cdot [g,u, g]^u. \]
This forces $[g,u,g] = 1$. By Lemma \ref{prop comutadores}\eqref{comut inverso}, $[g,u,g] = 1$ if and only if $[u,g,g] = 1$. This completes the second step.

\textbf{Step 3:} Finally, we prove that $G$ is a $2$-Engel group; that is, the closed characteristic subgroup \(R(G) = \{x \in G \mid [x,g,g] = 1 \text{ for all } g \in G\}\) coincides with $G$ (see Proposition \ref{caracteristico}). 

From Step 2, we know that all infinite-order elements of $G$ reside in $R(G)$. It remains to examine the finite-order elements. 

Let $v \in G$ be an element of finite order, and let $u \in G$ be an element of infinite order (which must exist since $G$ is non-periodic). As established in Step 2, the product $uv$ has infinite order. Therefore, the elements $u$, $u^{-1}$, $uv$, and $(uv)^{-1}$ all have infinite order and thus belong to $R(G)$. Since $R(G)$ is a subgroup, $v = u^{-1}(uv) \in R(G)$. This implies that $[v,g,g] = 1$ for all $g \in G$. Since every element of $G$ is in $R(G)$, $G$ is a $2$-Engel group.
\end{proof}

\begin{corollary} 
\label{mais geral G2engeldemais} 
Let $G \in \widehat{\mathscr{T}}$ be  a profinite group whose nontrivial \(p\)-Sylow subgroups are torsion-free.  Then $G$ is an \(2\)-Engel group. 
\end{corollary}
\begin{proof} 
First, \(G\) has class at most \(3\) by Lemma \ref{class3}. Hence \(G\) is pronilpotent, and by Proposition~2.3.8 in \cite{RZ10}, \(G\) is the direct product of its \(p\)-Sylow subgroups, say \(G = \displaystyle\prod_{p} G_p\), where \(G_p\) denotes the \(p\)-Sylow subgroup of \(G\). If \(G=\{1\}\), there is nothing to prove. Assume \(G \neq \{1\}\). Then some
Sylow subgroup \(G_p \neq \{1\}\), which is torsion-free, hence infinite,
hence contains an element \(u\) of infinite order, and therefore so does
\(G\). 

Let \(w \in G\) have infinite order. Since \(\overline{\langle w \rangle}\) is nilpotent (see Proposition~4.2 of \cite{LIM}), it is pronilpotent. Hence \(\overline{\langle w \rangle}\) is the direct product of its \(p\)-Sylow subgroups, say \(\overline{\langle w \rangle} = \displaystyle\prod_{p \in \pi'} W_p\), where \(\pi'\) is a nonempty set of primes and \(W_p\) denotes the \(p\)-Sylow subgroup of \(\overline{\langle w \rangle}\). By Corollary~2.3.6 in \cite{RZ10}, \(W_p \leq G_p\). Since the nontrivial \(p\)-Sylow subgroups of \(G\) are torsion-free, each \(W_p\) is torsion-free. Therefore \(\overline{\langle w \rangle}\) is torsion-free, and hence \(\overline{\langle w \rangle} \cong \displaystyle\prod_{p \in \pi'} \mathbb{Z}_{p}\). The result follows from Proposition~\ref{G2engeldemais}.  
\end{proof}

\begin{remark}
The hypothesis that all infinite procyclic subgroups are torsion-free was
crucial in the proof of Proposition~\ref{G2engeldemais}: without it, the
argument that a non-periodic profinite group all of whose infinite procyclic
subgroups are \(2\)-subnormal is \(2\)-Engel would not go through. This
contrasts with the abstract case, where Theorem~1 (1) of \cite{heineken}
does not require any restriction on the torsion of cyclic subgroups. It is
not known, however, whether this hypothesis can be removed in the profinite
case: there is neither a counterexample nor a proof that it is unnecessary.
\end{remark} 

The following result gives a profinite version of a theorem of Heineken on
torsion-free \(2\)-Baer groups (see \cite{heineken} and \cite{kappetraus}).  

\begin{theorem} \label{2baer} A torsion-free profinite group $G$ is $2$-Engel if and only if
$G\in\widehat{\mathscr T}$.     
\end{theorem}

\begin{proof}
This follows immediately from Lemma~\ref{lemma1} and
Proposition~\ref{G2engeldemais}, since $[g,h,h]=1$, and hence
$[g,h,h]\in\overline{\langle h\rangle}$, for all $g,h\in G$.    
\end{proof}   

\begin{corollary} \label{caso prop}
A non-periodic pro-\(p\) group \(G\) whose infinite procyclic subgroups are
torsion-free is \(2\)-Engel if and only if \(G\in\widehat{\mathscr{T}}\). 
\end{corollary}

\begin{proof} The proof uses the same arguments as that of Theorem~\ref{2baer}. 
\end{proof}

\begin{example}\label{grupo Q16}
Consider the generalized quaternion group of order $16$ (identified as
\texttt{SmallGroup(16,9)} in \cite{GAP}), defined by
$$
Q_{16}=\langle a,b\mid a^8=1,\; b^2=a^4,\; a^b=a^{-1}\rangle.
$$
It is nilpotent of class $3$, with center \(\langle a^4\rangle\cong C_2\) and
derived subgroup \(\langle a^2\rangle\cong C_4\). The cyclic subgroups of order $4$, generated by $b,ba,ba^2,ba^3$, are
$2$-subnormal, while all other subgroups are normal; hence every cyclic
subgroup of $Q_{16}$ is $2$-subnormal. Therefore $Q_{16}$ belongs to
$\mathscr T$; endowed with the discrete topology, it is a profinite group
in $\widehat{\mathscr T}$. Since $a^b=a^{-1}$, we have $[a,b]=a^{-1}a^b=a^{-2}$, and hence
$$
[a,b,b]=[a^{-2},b]=(a^2)(a^{-2})^b=a^2a^2=a^4\neq 1.
$$
Therefore $Q_{16}$ is not a $2$-Engel group.

Similarly, by Theorem~1 of~\cite{MAHDAVIANARY1984}, the group $H(3,1)$ with the presentation
$$
H(3,1) = \langle x, y \mid [y,x,x] = x^3 = [y,x,y] = y^3,\; x^9 = 1,\; \gamma_4^{abs}(H(3,1)) = \{1\} \rangle
$$
is a finite $3$-group of order $81$. Moreover, $H(3,1)$ is nilpotent of class $3$ and belongs to the class $\mathscr T$; when endowed with the discrete topology, it belongs to the class \(\widehat{\mathscr{T}}\). Since \([y,x,x] = x^3 \neq 1\), it follows that \(H(3,1)\) is not \(2\)-Engel.

Thus the non-periodicity hypothesis in Proposition~\ref{G2engeldemais} cannot be
removed: the examples $Q_{16}$ and $H(3,1)$ are periodic, belong to $\widehat{\mathscr{T}}$, and are
not $2$-Engel. Note also that $Q_{16}$ has an involution (namely $a^4$),
whereas $H(3,1)$, being a $3$-group, has none. Finally, these examples show
that the torsion-free hypothesis in Theorem~\ref{2baer} cannot be relaxed
either: for instance, \(H(3,1)\) lies in \(\widehat{\mathscr T}\) but is not
\(2\)-Engel.
\(\blacksquare\)
\end{example}

Let $G$ be a profinite group. For an integer $n \geq 1$, consider the continuous map \(\varphi_n\colon G \to G\), \(\varphi_n(x) = x^n\) \((x \in G)\), and let \(T_n = \varphi_n^{-1}(\{1\}) = \{x \in G : x^n = 1\}\) denote the set of \(n\)-th roots of unity in \(G\). Then \(T_n = \varphi_n^{-1}(\{1\})\) is closed in \(G\).

The following result is well-known for abstract groups (see Robinson \cite[Theorem~7.14]{chapter7}). The proof relies on the fact that the group satisfies the \(2\)-Engel condition, which we established for profinite groups in Proposition \ref{G2engeldemais}.

\begin{corollary}[\textnormal{Profinite version of Theorem~1 (2) in \cite{heineken}}] \label{gammatrestres}  
Let \(G \in \widehat{\mathscr{T}}\) be a non-periodic profinite group all of
whose infinite procyclic subgroups are torsion-free. Then \((\gamma_3(G))^3 = \{1\}.\) In particular, if \(G \in \widehat{\mathscr{T}}\) is torsion-free, the result applies. 
\end{corollary}
\begin{proof} 
We invoke the Hall-Witt identity (Lemma \ref{prop comutadores} \textit{(8)} ):
\begin{equation} \label{witthall1}
    1 = [x, y^{-1}, z]^{y} \cdot [y, z^{-1}, x]^{z} \cdot [z, x^{-1}, y]^{x} \quad \text{for all } x,y,z \in G. 
\end{equation}
By Proposition \ref{G2engeldemais}, $G$ is a $2$-Engel group. Applying Lemma \ref{identitiesof2engels} (ii) and (iii) along with Lemma \ref{prop comutadores}\eqref{comut produto dir}, we simplify the individual terms of the Hall-Witt identity:
\begin{align*}
    [x,y^{-1}, z]^{y} &= ([x,z,y^{-1}]^y)^{-1} = ([y,[x,z]])^{-1} = [x,z,y] = [x,y,z]^{-1}, \\
    [y,z^{-1},x]^z &= ([y,x,z^{-1}]^z)^{-1} = [z,[y,x]]^{-1} = [y,x,z] = [y,z,x]^{-1} = [x,[y,z]] = [x,y,z]^2, \\
    [z,x^{-1},y]^{x} &= ([z,y,x^{-1}]^x)^{-1} = [x,[z,y]]^{-1} = [x,z,y]^{-2} = [x,y,z]^2.
\end{align*}
Substituting these simplified terms back into equation \eqref{witthall1} yields:
\[ 1 = [x,y,z]^{-1} \cdot [x,y,z]^2 \cdot [x,y,z]^2 = [x,y,z]^3 \quad \text{for all } x,y,z \in G. \]

Since $G$ has class at most $3$ by Lemma \ref{class3}, we have 
$\gamma_4(G) = \{1\}$; hence $[\gamma_3(G), G] = \{1\}$ and $\gamma_3(G)$ is 
central. We claim that $\gamma_3(G) = \overline{\gamma_3^{abs}(G)}$. Indeed, 
$\gamma_3^{abs}(G) \leq \gamma_3(G)$ is immediate. For the reverse, let 
$N \trianglelefteq_o G$ be an open normal subgroup. In the finite quotient 
$G/N$, the closed and abstract central series coincide, so 
$\gamma_3(G)N/N = \gamma_3^{abs}(G)N/N$, hence 
$\gamma_3(G) \leq \gamma_3^{abs}(G)N$. Taking the intersection over all such 
$N$ and using 
$\overline{\gamma_3^{abs}(G)} = \bigcap_{N \trianglelefteq_o G} 
\gamma_3^{abs}(G)N$ \cite[Proposition~0.3.3]{wilson98}, we obtain 
\(\gamma_3(G) = \overline{\gamma_3^{abs}(G)}\). Thus \(\gamma_3(G)\) is 
topologically generated by commutators \([x,y,z]\), with no commutators of weight 
\(\geq 4\) to consider. Recall that \(T_3 = \{x \in G : x^3 = 1\}\) is closed in \(G\). Since 
\([x,y,z]^3 = 1\) for all \(x,y,z \in G\), we have \(\{[x,y,z]\mid x, y, z \in G\} \subseteq T_3\). As 
\(\gamma_3(G)\) is central, it is abelian, so the map \(w \mapsto w^3\) is a 
homomorphism on \(\gamma_3(G)\), and \(T_3 \cap \gamma_3(G)\) is a closed subgroup 
of \(\gamma_3(G)\) containing all generators \([x,y,z]\). Hence 
\(\gamma_3(G) = \overline{\langle \{[x,y,z]\mid x, y, z \in G\} \rangle} \subseteq T_3 \cap 
\gamma_3(G)\), and therefore \((\gamma_3(G))^3 = \{1\}\). 
\end{proof}

Note that, in the proof of Corollary \ref{gammatrestres}, the hypothesis on 
infinite procyclic subgroups serves only to deduce that \(G\) is 2-Engel: 
the non-periodicity of \(G\) guarantees, via Proposition 
\ref{G2engeldemais}, the existence of an element of infinite order whose 
procyclic subgroup is torsion-free, which in turn yields the 2-Engel 
condition. Once 2-Engelness is assumed outright, the non-periodicity 
hypothesis becomes superfluous, and the same conclusion holds, as recorded 
below.

\begin{proposition}\label{gammatrestres2engel}
Let \(G \in \widehat{\mathscr{T}}\) be a 2-Engel profinite group. Then 
\((\gamma_3(G))^3 = \{1\}\).
\end{proposition}

\begin{proof}
Since \(G\) is 2-Engel, \([x,y,y]=1\) for all \(x,y \in G\). Combining this 
with the standard commutator identities yields \([x,y,z]^3=1\) for all 
\(x,y,z \in G\) (the same computation as in the proof of Corollary 
\ref{gammatrestres}, \textit{mutatis mutandis}). As \(\gamma_3(G)\) is generated by 
commutators of weight \(3\), it follows that \((\gamma_3(G))^3=\{1\}\).
\end{proof} 

The following result answers Question \ref{korou} for torsion-free profinite groups with \(n=2\). For \(n=1\), such groups are Dedekind, hence profinite abelian, and thus of class \(1\).

\begin{theorem} \label{leviprofinite} 
Let \(G\) be a non-periodic pro-\(\pi\) group without elements of order
\(3\), whose infinite procyclic subgroups are torsion-free, and such that
every closed subgroup of \(G\) is \(2\)-subnormal. Then its nilpotency
class is at most \(2\). In particular, this holds if \(G\) is a torsion-free
profinite group all of whose closed subgroups are \(2\)-subnormal.
\end{theorem}
\begin{proof} As every closed subgroup is \(2\)-subnormal, so is every cyclic subgroup, and therefore \(G \in \widehat{\mathscr{T}}.\) By Corollary \ref{gammatrestres}, we have \((\gamma_3(G))^3 = \{1\}\). Since \(G\) has no elements of order \(3\), this forces
\(\gamma_3(G) = \{1\}\), so \(G\) has nilpotency class at most \(2\). 
\end{proof}

The following result is a profinite version of a theorem of Stadelmann
(see \cite{Stadelmann1978}). 

\begin{corollary} \label{freetorsion classe2}
Let $G$ be a torsion-free profinite group. Every closed subgroup of $G$ has
subnormal defect at most $2$ if and only if $G$ is nilpotent of class $2$.     
\end{corollary}
\begin{proof}
It follows directly from Proposition~\ref{defecttwo} and Theorem~\ref{leviprofinite}. 
\end{proof}

\begin{remark}
The group \(H(3,1)\) from Example~\ref{grupo Q16} shows that the
torsion-free hypothesis in Corollary~\ref{freetorsion classe2} cannot be
relaxed. By Corollary~1 of~\cite{MAHDAVIANARY1984}, for a \(2\)-generator
\(p\)-group with \(p\geq 3\), membership in \(\mathscr S\) is equivalent to
membership in \(\mathscr T\); equivalently, when such a finite group is
endowed with the discrete topology, it lies in the hatted class
\(\widehat{\mathscr T}\) if and only if it lies in \(\widehat{\mathscr S}\).
In particular, \(H(3,1)\in\widehat{\mathscr S}\), so every subgroup of
\(H(3,1)\) has defect at most \(2\); yet \(H(3,1)\) is nilpotent of class \(3\).
\end{remark}

\begin{proposition} \label{prop1} Let $G \in \widehat{\mathscr{T}}$ be a torsion profinite group without involutions. Then $G \in \widehat{\mathscr{A}}$.
\end{proposition}
\begin{proof}
By \cite[Theorem~2.1.3]{RZ10}, $G$ can be written as \(G = \varprojlim_{U \in \mathcal{N}} G/U\), where \(\mathcal{N}\) is a fundamental system of open neighborhoods of the identity \(1 \in G\) consisting of open normal subgroups of \(G\). For each \(U \in \mathcal{N}\), let \(\pi_{U} \colon G \twoheadrightarrow G/U\) be the continuous canonical projection, and these projections are compatible. 

Let $H$ be an arbitrary closed abelian subgroup of $G$. To prove that $G \in \widehat{\mathscr{A}}$, we must show that $H$ has subnormal defect at most $2$. Since $H$ is abelian, its image \(\pi_U(H) = HU/U\) is abelian for each \(U \in \mathcal{N}\). Because \(G\) is a periodic profinite group lacking even-order elements (a pro-\(2'\) group), every finite quotient \(G/U\) is a finite \(2'\)-group. By Lemma \ref{formula0}, \(G/U \in \mathscr{T}\) for all \(U \in \mathcal{N}\). Applying Mahdavianary's abstract result \cite[Theorem~1(i)]{mahdavimain}, any finite \(2'\)-group in \(\mathscr{T}\) necessarily belongs to \(\mathscr{A}\). Thus, \(G/U \in \mathscr{A}\) for all \(U \in \mathcal{N}\), implying that the abelian subgroup \(HU/U\) is \(2\)-subnormal in \(G/U\), i.e.\ \(HU/U \unlhd^{2} G/U\) for all \(U \in \mathcal{N}\). 
Since \(G \in \widehat{\mathscr{T}}\), Lemma \ref{class3} ensures that \(G \in \widehat{\mathfrak{N}}_{3}\). By Proposition \ref{defecttwo}, every closed subgroup of \(G\) has
defect at most \(3\). Suppose for a contradiction that \(H\) has defect
exactly \(3\). 
Then \(H\) admits the strict chain of normal closures:
\[ H = H_3 = \overline{H^{H_2}} \triangleleft_c H_2 = \overline{H^{\overline{H^G}}} \triangleleft_c H_1 = \overline{H^G} \triangleleft_c H_0 = G. \]
On the other hand, since \(HU/U\) has defect at most \(2\) in \(G/U\), we have
\[ \pi_U(H)=\frac{HU}{U} \unlhd \pi_U(\overline{H^G}) =\pi_U(H^G)= \frac{H^G U}{U} =\left( \frac{HU}{U} \right)^{G/U}. \]
Hence, by the Correspondence Theorem, pulling back through \(\pi_U\) gives \(HU \unlhd_c \overline{H^G}U\). Since this holds for every \(U \in \mathcal{N}\), it follows that \(H \unlhd_c \overline{H^G} = H_1\), contradicting the strictness of the chain \(H \triangleleft_c H_2 \triangleleft_c H_1\). Therefore, the subnormal defect of \(H\) in \(G\) must be at most \(2\). This proves that \(G \in \widehat{\mathscr{A}}\). 
\end{proof}

\begin{corollary} \label{da finito}
Let $G$ be a finitely generated periodic pro-\(\pi\) group, where \(\pi\) is a finite set of odd primes. If all procyclic subgroups of \(G\) have a subnormal defect of at most 2. Then \(G\) is a finite \(\pi\)-group in which all abelian subgroups are subnormal of defect at most 2. Furthermore, if \(G\) is \(2\)-generated, it belongs to the class \(\mathscr{S}\), i.e.\ every subgroup of \(G\) has subnormal defect at most \(2\).
\end{corollary}
\begin{proof}
By hypothesis, \(G \in \widehat{\mathscr{T}}\). By Corollary \ref{cor4}, \(G\) is a nilpotent profinite group of class at most \(3\). As a pronilpotent group, \(G\) decomposes into the direct product of its Sylow \(p\)-subgroups (\cite[Proposition~2.3.8]{RZ10}). According to Zel'manov's celebrated solution to the Restricted Burnside Problem \cite{zelma}, any finitely generated periodic pro-\(p\) group must be finite. Therefore, every Sylow \(p\)-subgroup of \(G\) is finite. Since \(\pi\) is a finite set, the direct product of these Sylow subgroups is finite, meaning \(G\) itself is a finite \(\pi\)-group. Because \(G\) is a finite group lacking even-order elements, it falls into \(\mathscr{A}\) by Proposition \ref{prop1} (or \cite[Theorem~1(i)]{mahdavimain}). The final statement follows from the established result that for \(2\)-generated groups, the classes collapse to \(\mathscr{S} = \mathscr{A} = \mathscr{T}\) (see \cite{mahdavithesis}).
\end{proof}

\begin{proposition} \label{prop2} 
Let \(G \in \widehat{\mathscr{T}}\) be a non-periodic profinite group all of
whose infinite procyclic subgroups are torsion-free. Then \(G \in \widehat{\mathscr{A}}.\)
\end{proposition}
\begin{proof} 
By Lemma \ref{class3}, \(G\) is a nilpotent profinite group of class at most \(3\). In particular, \(G\) is metabelian and \(\gamma_3(G) \leqslant Z(G)\). In light of Lemma \ref{2geradob}, to prove that \(G \in \widehat{\mathscr{A}}\), it suffices to show that for any commuting elements \(x, y \in G\) and for all \(z \in G\), the commutator \([z,x,y]\) lies in \(H = \overline{\langle x,y \rangle}\).

Because \(\gamma_3(G)\) is central and \(H\) is abelian (\([x,y]=1\)), the Hall-Witt identity reduces to:
\begin{equation*} 
    1 = [x, z, y]^{z^{-1}} \cdot [z,y,x]^{y^{-1}} \cdot [y,x,z]^{x^{-1}} = [x,z,y] \cdot [z,y,x] \quad \text{for all } z \in G. 
\end{equation*}
This establishes that \([z,y,x] = [x,z,y]^{-1}\). Since \(G\) is metabelian, the derived subgroup \(\gamma_2(G) = \overline{[G,G]}\) is abelian. Because the conjugate \(([x,z]^{-1})^{y}\) lies in \(\gamma_2(G)\), we evaluate the commutator:
\begin{align} \label{quasela}
    [z,y,x] &= [x,z,y]^{-1} = [y, [x,z]] = ([x,z]^{-1})^{y}[x,z] \nonumber \\
    &= [x,z]([x,z]^{-1})^{y} = [[x,z]^{-1},y] = [z,x,y].
\end{align}

By Lemma \ref{lemma1}, we know that \([z,xy,xy] \in \overline{\langle xy \rangle} \leq H\). Since \(\gamma_3(G)\) is central, the following expansion holds for any \(a,b,c,d \in G\):
\begin{equation} \label{usaagora}
    [[a,b],cd] = [[a,b],d] \cdot [[a,b],c]^d = [[a,b],d] \cdot [[a,b],c]. 
\end{equation} 
As \(\gamma_4(G) = \{1\}\) and \(\gamma_3(G)\) is central, repeated application of \eqref{usaagora} and Lemma \ref{prop comutadores} to expand \([z,xy,xy]\):
\begin{align*}
    [z,xy,xy] &= [[z,xy],y] \cdot [[z,xy],x] \\
    &= [[z,y][z,x][z,x,y], y] \cdot [[z,y][z,x][z,x,y], x] \\
    &= [[z,y][z,x], y] \cdot [[z,y][z,x], x] \\
    &= [[z,y], y] \cdot [[z,x], y] \cdot [[z,y], x] \cdot [[z,x], x] \\
    &= [z,y,y] \cdot [z,x,y] \cdot [z,y,x] \cdot [z,x,x] \\
    &= [z,x,x] \cdot [z,y,y] \cdot [z,x,y] \cdot [z,y,x].
\end{align*}
By Lemma \ref{lemma1}, \([z,x,x] \in \overline{\langle x \rangle} \leq H\) and \([z,y,y] \in \overline{\langle y \rangle} \leq H\). Thus, their product \([z,x,x][z,y,y]\) belongs to \(H\). Utilizing equation \eqref{quasela}, which states \([z,y,x] = [z,x,y]\), we can isolate the cross terms:
\[ [z,x,y]^2 = [z,y,y]^{-1} [z,x,x]^{-1} [z,xy,xy] \in \overline{\langle x,y \rangle} = H. \]

Finally, by Corollary \ref{gammatrestres}, \((\gamma_3(G))^3 = \{1\}\). This allows us to write the inverse as a square:
\[ 1 = [z,x,y]^3 = [z,x,y]^2 [z,x,y] \implies [z,x,y]^{-1} = [z,x,y]^{2} \in H. \]
Consequently, \([z,x,y] \in H\). By Lemma \ref{2geradob}, \(G \in \widehat{\mathscr{A}}\).
\end{proof}

\begin{proposition}\label{prop2soengel}
Let \(G \in \widehat{\mathscr{T}}\) be a non-periodic 2-Engel profinite 
group. Then \(G \in \widehat{\mathscr{A}}\).
\end{proposition}

\begin{proof}
The proof is identical to that of Proposition \ref{prop2}, except 
that the step invoking \((\gamma_3(G))^3=\{1\}\) now relies on Proposition 
\ref{gammatrestres2engel} rather than on the hypothesis on infinite 
procyclic subgroups. Indeed, in Proposition \ref{prop2} that hypothesis 
serves only to show that \(G\) is 2-Engel and, consequently, that 
\((\gamma_3(G))^3=\{1\}\); here this is precisely what Proposition 
\ref{gammatrestres2engel} provides, and all remaining arguments are 
unchanged.
\end{proof} 

We summarize our last results in the profinite analogue of Mahdavianary's \cite[Theorem~1(i)]{mahdavimain}.

\begin{theoremalfa}\label{maintheorem}
If \(G \in \widehat{\mathscr{T}}\) is a profinite group which either is
non-periodic and has all its infinite procyclic subgroups torsion-free, or
is torsion and has no involutions, then \(G \in \widehat{\mathscr{A}}\).    
\end{theoremalfa} 
\begin{proof}
This follows directly from Propositions \ref{prop1} and \ref{prop2}.    
\end{proof}

\begin{corollary}
Let $G$ be a torsion-free profinite group all of whose procyclic subgroups are $2$-subnormal. Then $G \in \widehat{\mathscr{A}}.$  
\end{corollary}

\section{Strict inclusions and Cartesian products}

Although the strictness of the inclusions \(\widehat{\mathscr{S}} \subseteq \widehat{\mathscr{A}} \subseteq \widehat{\mathscr{T}}\) can be partially resolved when looking at finite groups endowed with the discrete topology, it is a natural question whether these classes remain strictly distinct when considering infinite profinite groups. In this final section, we extend Mahdavianary's discrete counterexamples \cite{mahdavimain} to construct explicit infinite profinite groups that separate these classes. Furthermore, we prove that the class \(\widehat{\mathscr{T}}\) is not closed under arbitrary Cartesian products.

We begin by recalling the properties of a specific finite group constructed by Mahdavianary \cite[Theorem~1(ii)]{mahdavimain}.

\begin{remark} \label{grupo D caracterizado}
Let $D$ be the group defined by the presentation:
\begin{align*}
    D =  \langle a, b, c \mid \; &[x_1, x_2, x_2]=1 \text{ for } x_1, x_2 \in \{a, b, c\}, \  c^2=b^8=1, [a, c, b]=[a, c]=1,\\
    &[b, c, a] = [a, b, c]=b^4, [a, b]^2 = [b, c]^2=a^2=1, \ \gamma_4(D)=1 \rangle.
\end{align*} 
Using the relations of $D$ (which has order $128$ and corresponds to \texttt{SmallGroup(128, 523)} in the GAP library \cite{GAP}), one can verify that $D$ is a nilpotent group of class $3$, hence metabelian, with center \(Z(D) = \langle b^2 \rangle \cong C_4\), derived subgroup \(D' = \langle [a, b], [b, c], b^4 \rangle\), and exponent \(8\). Mahdavianary proved that \(D \in \mathscr{T}\), but \(D \notin \mathscr{A}\) (since the elementary abelian subgroup \(\langle a, c \rangle\) has subnormal defect 3 in \(D\)). Notably, \(D\) is not a \(2\)-Engel group, despite the \(2\)-Engel identity holding for its generators, as evidenced by \([c, ab, ab] = b^4 \neq 1\).
\end{remark} 

\begin{theorem} \label{teo_infinite_counterexample_D}
There exists an infinite periodic profinite group $V$ such that $V \in \widehat{\mathscr{T}}$ but $V \notin \widehat{\mathscr{A}}$. 
\end{theorem}
\begin{proof} 
Let \(V = D \times E\) be the direct product of the Mahdavianary group \(D\) (see Remark \ref{grupo D caracterizado}) and an infinite elementary abelian pro-\(2\) group \(E\). Because \(D\) has exponent 8 and \(E\) has exponent 2, \(V\) is an infinitely generated periodic pro-\(2\) group of exponent 8. Observe that for any \(u \in V\), the procyclic subgroup \(\overline{\langle u \rangle} = \langle u \rangle\) is a finite \(2\)-group of order at most 8. We divide the remainder of the proof into two steps.

First, we show that \(V \in \widehat{\mathscr{T}}\). By Lemma \ref{lemma1}, \(V \in \widehat{\mathscr{T}}\) if and only if \([u,v,v] \in \langle v \rangle\) for all \(u, v \in V\). Let \(u = (d_1, u_1)\) and \(v = (d_2, u_2)\) be arbitrary elements in \(D \times E\). Since \(E\) is abelian, we have \([u,v,v] = ([d_1, d_2, d_2], 1)\). Since \(D \in \mathscr{T}\), Lemma~1 of \cite{mahdavimain} gives \([d_1, d_2, d_2] \in \langle d_2 \rangle\). If \(d_2=1\), then \([u,v,v] = 1 \in \langle v \rangle\). If \(d_2 \neq 1\), since \(\gamma_4^{abs}(D)=\{1\}\), we have
\(\gamma_3^{abs}(D) \leqslant Z(D) = \langle b^2 \rangle\), hence
\([d_1,d_2,d_2] \in Z(D) \cap \langle d_2 \rangle\).
If \(d_2 \in Z(D)\), then \([d_1,d_2,d_2]=1=d_2^0\); otherwise
\(Z(D) \cap \langle d_2 \rangle\) is a proper subgroup of \(\langle d_2 \rangle\),
hence contained in \(\langle d_2^2 \rangle\). Therefore
\([d_1,d_2,d_2] = d_2^{2k}\) for some integer \(0 \leq k \leq 3.\)  Consequently, 
\[ [u,v,v] = ([d_1, d_2, d_2], 1)= (d_2^{2k}, 1) = (d_2^{2k}, u_2^{2k}) = v^{2k} \in \langle v \rangle. \]
Therefore, \(V \in \widehat{\mathscr{T}}\).

Second, we prove that \(V \notin \widehat{\mathscr{A}}\). Consider the subgroup \(H = \langle a, c \rangle \leq D\). By \cite[Theorem~1(ii)]{mahdavimain}, \(H\) is an elementary abelian \(2\)-subgroup of \(D\) whose subnormal defect is exactly 3. We view \(H\) as a closed abelian subgroup of \(V\) (via \(H \times \{1\}\)). Suppose, for a contradiction, that the subnormal defect of \(H\) in \(V\) is at most \(2\). Then there exists a closed normal chain \(H \unlhd_c L \unlhd_c V\). Intersecting this chain with the discrete subgroup \(D \times \{1\} \cong D\), we obtain the chain
\(H \unlhd (L \cap D) \unlhd D\).
This implies that \(H\) has subnormal defect at most \(2\) in \(D\), contradicting the fact that its subnormal defect in \(D\) is \(3\). Thus, the subnormal defect of \(H\) in \(V\) is strictly greater than 2, proving that \(V \notin \widehat{\mathscr{A}}\).
\end{proof}

\begin{remark} \label{potencia gama} A well-known fact about an abstract nilpotent group \(G\), which can be proved by induction on \(i\) using Lemma \ref{prop comutadores} (5), is that for every \(n \in \mathbb{N}\), if \(a \in \gamma_{i}^{abs}(G)\) and \(b \in G\), then \([a, b^n] \equiv [a, b]^n \pmod{\gamma_{i+2}^{abs}(G)}\) (see \cite{Hall1933}).  
\end{remark}

\begin{theorem} \label{teo_infinite_counterexample_cappitt}
There exists a non-periodic 2-Engel profinite group \(W\) without 
involutions such that \(W \in \widehat{\mathscr{A}}\) but 
\(W \notin \widehat{\mathscr{S}}\). 
\end{theorem}
\begin{proof}
Consider the abstract group \(G\) defined by Cappitt \cite{cappitt2} and further analyzed by Stadelmann \cite{Stadelmann1978} and Mahdavianary \cite[Theorem~1(iii)]{mahdavimain}, given by the presentation:
\begin{equation*}  
    G = \langle x,y,z \mid [u,v,v]=1 \text{ for all } u,v \in G, \; y^3 = [z,y], \; z^9 = [y,x]^3 = [z,y,x], \; [z,x]^3 = 1 \rangle.
\end{equation*} 
By definition, \(G\) is a \(2\)-Engel group. By Levi's Theorem \cite{levi}, any \(2\)-Engel group is nilpotent of class at most \(3\). Moreover, the relation \([z, y, x] = z^9 \neq 1\) shows that \(\gamma_3(G) \neq \{1\}\), so \(G\) has class exactly \(3\). Because \(G\) is a finitely generated nilpotent group, it is polycyclic and hence has the ERF property. Furthermore, its abelianization \(G^{ab}=G/G'\) is isomorphic to \(\mathbb{Z} \times C_3 \times C_9\), meaning \(G\) is an infinite group. Let \(W = \widehat{G}\) be its profinite completion. Note that the completion functor is right-exact (see Proposition~3.2.5 in \cite{RZ10}); since \(G\) is polycyclic,
\[ W^{ab} = \left( \widehat{G} \right)^{ab} = \frac{\widehat{G}}{\overline{[\widehat{G}, \widehat{G}]}} \cong \widehat{\left( \frac{G}{[G,G]} \right)} = \widehat{G^{ab}} \cong \widehat{\mathbb{Z}} \times C_3 \times C_9. \]
Therefore the profinite group \(W\) is non-periodic. 

\medskip
We now divide the remainder of the proof into three parts.

\medskip

\noindent\textbf{Part 1.} We prove that \(W\in\widehat{\mathscr{A}}\). The \(2\)-Engel identity \([u,v,v]=1\) is preserved by all finite quotients of \(G\), and consequently by the inverse limit. Thus, \(W = \widehat{G}\) is a \(2\)-Engel profinite group. By Lemma \ref{lemma1}, \(W \in \widehat{\mathscr{T}}\). Therefore \(W\) is a non-periodic 2-Engel profinite group, and it follows 
from Proposition \ref{prop2soengel} that \(W \in \widehat{\mathscr{A}}\). 

\medskip
\noindent\textbf{Part 2.} Next, we prove that \(W \notin \widehat{\mathscr{S}}\). Consider the subgroup \(H = \langle x^3z, x^9z \rangle \leq G\). Stadelmann \cite{Stadelmann1978} showed that \(H\) has subnormal defect strictly greater than 2 in \(G\). Let \(\overline{H}\) be the topological closure of \(H\) in \(W\). Suppose for a contradiction that \(\overline{H}\) is \(2\)-subnormal in \(W\), 
which implies a closed normal chain 
\(\overline{H} \triangleleft_c \overline{\overline{H}^{W}} \triangleleft_c W=\widehat{G}\). Intersecting this chain with \(G\) yields:
\[ (\overline{H} \cap G) \triangleleft \left(\overline{\overline{H}^{W}} \cap G\right) \triangleleft G. \]
Since \(G\) is polycyclic (hence ERF), Lemma \ref{lem_polycyclic_normal_closure} guarantees that \(\overline{H}^{W} \cap G=\overline{\overline{H}^{W}} \cap G = H^G\). Moreover, since \(H\) is finitely generated, it is closed in the profinite topology of \(G\), which means that \(\overline{H} \cap G = H\). Substituting these identities, we obtain the abstract subnormal chain \(H \triangleleft H^G \triangleleft G\). This implies that \(H\) has subnormal defect at most 2 in \(G\), a contradiction. Therefore, \(\overline{H}\) has defect strictly greater than 2 in \(W\), and we prove that \(W=\widehat{G} \notin \widehat{\mathscr{S}}\). 

\medskip
\noindent\textbf{Part 3.} Finally, it remains to show that \(W\), and also \(G\), have no involutions. 
Before doing so, we first establish the following. 

\medskip
\noindent\textbf{Claim.} \(\gamma_2^{abs}(G)=[G,G]\) is a finite \(3\)-group.

\medskip

\noindent Indeed, since \(G\) has class \(3\),
\[
[\gamma_2^{abs}(G),\gamma_2^{abs}(G)]\leq \gamma_4^{abs}(G)=\{1\}
\]
by \cite[Proposition 5.1.11]{robinson}. Thus \(\gamma_2^{abs}(G)\) is abelian, and so \(\gamma_3^{abs}(G)\) is an abelian normal subgroup of \(\gamma_2^{abs}(G)\). Since \(G\in\mathscr{T}\) by \cite[Theorem I (iii)]{mahdavimain}, it follows from \cite[Theorem 1]{heineken} that \((\gamma_3^{abs}(G))^3=\{1\}.\) Moreover, \(\gamma_3^{abs}(G)\) is finitely generated, since it is generated by the commutators of weight \(3\) in the generators \(x,y,z\) of \(G\) \cite[Lemma 2.6]{clement2017theory}. Hence \(\gamma_3^{abs}(G)\) is a finite elementary abelian \(3\)-group. By Theorem 3.1 of \cite{clement2017theory}, the quotient \(\gamma_2^{abs}(G)/\gamma_3^{abs}(G)\) is generated by the images of \([x,y]\), \([y,z]\), \([z,x]\) modulo \(\gamma_3^{abs}(G)\). Using the relations of the presentation of \(G\) and Remark \ref{potencia gama}, we have
\[
[z,x]^3=1,\qquad [y,x]^3=[z,y,x]\equiv 1 \pmod{\gamma_3^{abs}(G)}.
\]
Also, since \(y^3=[z,y]\in\gamma_2^{abs}(G)\), we have \([y^3,z]\in\gamma_3^{abs}(G)\). By Remark \ref{potencia gama}, applied with \(a=y\), \(b=z\) and \(i=1\),
\[
[y,z]^3\equiv [y^3,z]=[z,y,z]\equiv 1 \pmod{\gamma_3^{abs}(G)}.
\]
Hence every generator of \(\gamma_2^{abs}(G)/\gamma_3^{abs}(G)\) has order dividing \(3\). Since the quotient is abelian and finitely generated,
\[
\gamma_2^{abs}(G)/\gamma_3^{abs}(G)\cong C_3^{\,s}
\]
for some \(s\ge 0\), so it is a finite elementary abelian \(3\)-group. Since \(\gamma_3^{abs}(G)\) and \(\gamma_2^{abs}(G)/\gamma_3^{abs}(G)\) are both finite \(3\)-groups, it follows that \(\gamma_2^{abs}(G)\) is a finite \(3\)-group, which proves the claim. 

\bigskip

With the claim established, we now proceed to prove that \(W=\widehat{G}\) 
has no involutions. The commutator map
\[
[\cdot,\cdot]\colon W\times W\to W,\qquad (a,b)\mapsto[a,b],
\]
is continuous, and \(G\times G\) is dense in \(W\times W\). Hence
\[
\overline{[G,G]}=\overline{[W,W]}=[W,W],
\]
where the last equality uses that \([W,W]\) is closed in \(W\) (Theorem 4.8.2 in \cite{RZ10}). Since \(\gamma_2^{abs}(G)=[G,G]\) is finite, it is closed in \(W\), so
\[
[W,W]=\overline{[G,G]}=[G,G]=\gamma_2^{abs}(G).
\]
Therefore \([W,W]\) is a finite \(3\)-group. In particular, it has no elements of order \(2\). Now suppose, for a contradiction, that \(W\) has an involution \(w\). Its image in \(W^{ab}\cong \widehat{\mathbb Z}\times C_3\times C_9\) has order dividing \(2\). Since \(W^{ab}\) has no \(2\)-torsion, \(w\in\gamma_2(W)\). But \(\gamma_2(W)\) is a \(3\)-group and \(w^2=1\), so \(w=1\), a contradiction. Hence \(W=\widehat{G}\) has no involutions, and therefore neither does \(G\). 

\end{proof}

In order to prove that the class \(\widehat{\mathscr{T}}\) is not closed under Cartesian products, 
we require two preliminary lemmas detailing the behavior of elements in the 
Mahdavianary group \(D\).

\begin{lemma} \label{ordem de ab}
In the group $D$ (defined in Remark \ref{grupo D caracterizado}), $(ab)^4 = b^4$ and $|ab| = 8$.    
\end{lemma}
\begin{proof}
Recall from Remark \ref{grupo D caracterizado} that from the relations in \(D=\left\langle a,b,c\right\rangle\), 
we have \(a^2 = b^8 = [a,b]^2 = 1\), \(Z(D) = \langle b^2 \rangle\), and 
\([x_1, x_2, x_2] = 1\) whenever \(x_1, x_2 \in \{a,b,c\}\). The last equality implies \([b,a,a] = [a,b,a] = [a,b,b] = 1\), 
from which it follows that \([a,b]\) commutes with both \(a\) and \(b\). Making use of Lemma \ref{prop comutadores}\eqref{comut conjugado}, we expand \((ab)^2\):
\[ (ab)^2 = abab = a^2 (a^{-1}bab) = 1 \cdot b^a b = b[b,a]b = b^2[b,a]. \]
Squaring this result yields:
\[ (ab)^4 = \left( b^2[b,a] \right)^2 = b^4 [b,a]^2 = b^4 \cdot 1 = b^4. \]
Since \(b^8 = 1\) and \(b^4 \neq 1\), it follows that \((ab)^4 \neq 1\) and \((ab)^8 = 1\), making \(|ab| = 8\).
\end{proof}

\begin{lemma} \label{comutador de c b ab}
In the group $D$, $[c, b, ab] = b^4$.    
\end{lemma}
\begin{proof}
Applying Lemma \ref{prop comutadores}\eqref{comut produto dir}, we expand the commutator:
\[ [c,b,ab] = [[c, b], ab] = [[c, b], b] \cdot [[c, b], a]^b. \]
Since \([x_1, x_2, x_2]=1\) for the generators, \([c,b,b] = 1\). Thus, the first term vanishes. For the second term, observe that \([b,c]^2=1\) implies \([c,b] = [b,c]^{-1} = [b,c]\). Therefore, \([[c,b],a] = [b,c,a] = b^4\). Because \([D,D,D] \leqslant Z(D)\), \(b^4\) is central, which means its conjugate by \(b\) remains unchanged. Hence,
\[ [c,b,ab] = (b^4)^b = b^4. \]
\end{proof}

\begin{theorem} \label{T não é fechada}
The class $\widehat{\mathscr{T}}$ is not closed under arbitrary Cartesian products.    
\end{theorem}
\begin{proof}
Let \(U = \prod_{i \in \mathbb{N}} D_i\) be the infinite Cartesian product of copies of the group \(D\) (where \(D_i \cong D\) for all \(i\)). Since each \(D_i\) is finite, and every finite group in 
\(\widehat{\mathscr{T}}\), endowed with the discrete topology, lies in 
\(\mathscr{T}\), each \(D_i \in \widehat{\mathscr{T}}\) is also in 
\(\mathscr{T}\) (Remark \ref{grupo D caracterizado}). Let \(a_i\), \(b_i\), \(c_i\) denote the generators of \(D_i\). The group \(U\) is an infinitely generated periodic pro-\(2\) group of exponent 8. Consequently, the procyclic subgroup \(\overline{\langle u \rangle} = \langle u \rangle\) generated by any \(u \in U\) is a finite \(2\)-group of order at most 8. 

Define two elements \(h, g \in U\) as follows:
\begin{align*}
    h &= (a_1b_1, b_2, 1, 1, \dots), \\
    g &= (c_1, a_2, 1, 1, \dots).
\end{align*}
Suppose for a contradiction that \(U \in \widehat{\mathscr{T}}\). By Lemma \ref{lemma1}, \([g,h,h] \in \langle h \rangle\). Let us evaluate \([g,h,h]\) coordinate-wise:
\[ [g,h,h] = \left( [c_1, a_1b_1, a_1b_1], [a_2, b_2, b_2], 1, 1, \dots \right). \]
In the second coordinate, the identity \([x_1, x_2, x_2]=1\) immediately forces \([a_2, b_2, b_2] = 1\). In the first coordinate, we expand the inner commutator applying Lemma \ref{prop comutadores}\eqref{comut produto dir}:
\[ [c_1, a_1b_1] = [c_1, b_1] \cdot [c_1, a_1]^{b_1}. \]
Since \([c_1, a_1] = 1\), this reduces to \([c_1, a_1b_1] = [c_1, b_1]\). Reinserting this into the triple commutator yields:
\[ [c_1, a_1b_1, a_1b_1] = [[c_1, b_1], a_1b_1] = [c_1, b_1, a_1b_1]. \]
By Lemma \ref{comutador de c b ab}, \([c_1, b_1, a_1b_1] = b_1^4\). Therefore, \([g,h,h] = (b_1^4, 1, 1, 1, \dots)\). Since \([g,h,h] \in \langle h \rangle\), there exists an integer \(k\) 
such that \([g,h,h] = h^k\). Hence
\[ (b_1^4, 1, 1, \dots) = h^k = ((a_1b_1)^k, b_2^k, 1, \dots). \]
Comparing coordinates, we obtain \(b_2^k = 1\) and \((a_1b_1)^k = b_1^4\). Since \(|b_2| = 8\), the relation \(b_2^k = 1\) forces \(k\) to be a multiple of 8. However, by Lemma \ref{ordem de ab}, \(|a_1b_1| = 8\). Thus, \(k\) being a multiple of 8 forces \((a_1b_1)^k = 1\). This implies \(b_1^4 = 1\), which contradicts the presentation of \(D_1\) where \(b_1^4 \neq 1\). Therefore, we conclude that \([g,h,h] \notin \langle h \rangle\), 
which shows that \(U \notin \widehat{\mathscr{T}}\).
\end{proof}

Note that the group \(U\) is not contained in either \(\widehat{\mathscr{S}}\) or
\(\widehat{\mathscr{A}}\).
\begin{remark}\begin{sloppypar}\label{grupo Q16 caracterizado}
The failure of closure under Cartesian products is not restricted to
infinite products; it also occurs for finite direct products. Consider the
generalized quaternion group of order \(16\), defined by
$$
Q_{16}=\langle a,b\mid a^8=1,\; b^2=a^4,\; a^b=a^{-1}\rangle.
$$
By Example~\ref{grupo Q16}, \(Q_{16}\in\mathscr S\subset \mathscr A\subset 
\mathscr T\). Since \(Q_{16}\) is finite, when endowed with the discrete
topology it belongs to
\(\widehat{\mathscr{S}}\subset \widehat{\mathscr{A}}\subset 
\widehat{\mathscr{T}}\). 

Yet, \(Q_{16}\times Q_{16}\notin\mathscr{T}\). Suppose, for a contradiction,
that it were. Consider the elements \(g=(a,1)\) and \(h=(b,b)\) in
\(Q_{16}\times Q_{16}\). By \cite[Lemma~1]{mahdavimain}, we must have
\([g,h,h]\in\langle h\rangle\). Evaluating the commutator gives \([g,h,h]=([a,b,b],1)\). By Example~\ref{grupo Q16}, \([a,b,b]=a^4\neq 1\). Thus
\([g,h,h]=(a^4,1)\). However,
$$
\langle (b,b)\rangle
=
\{(1,1),\,(b,b),\,(a^4,a^4),\,(a^4b,a^4b)\},
$$
which does not contain \((a^4,1)\). Therefore \([g,h,h]\notin\langle h\rangle
=\langle (b,b)\rangle\), a contradiction. Consequently,
\(Q_{16}\times Q_{16}\notin\mathscr T\). This finite counterexample reaffirms
that the classes \(\widehat{\mathscr S}\), \(\widehat{\mathscr A}\), and
\(\widehat{\mathscr{T}}\) are not closed under Cartesian products.
\end{sloppypar}
\end{remark}

\section*{Acknowledgements}
The first author is grateful for the support from UFMG in his Post-Doctoral internship at this institution and for financial support from FAPDF. The second and third authors were partially supported by FAPDF and the last author is also grateful for the support from CNPq and for partial support from CAPES during her PhD.


\end{document}